\documentclass[10pt]{article}
\usepackage{amsmath,amsthm,amsfonts,amssymb}
\usepackage[letterpaper,margin=1in]{geometry}
\usepackage{graphicx}
\usepackage[numbers,sort]{natbib}
\usepackage[colorlinks,linkcolor=blue, citecolor=blue]{hyperref}
\usepackage{cleveref}
\usepackage{authblk}
\usepackage{bm}
\usepackage{bbm}
\usepackage{amsbsy}
\usepackage{enumitem}
\usepackage{float}
\usepackage{cases} 
\usepackage[dvipsnames]{xcolor}
\usepackage{subcaption}
\usepackage{tikz}
\usetikzlibrary{arrows.meta}
\usepackage{changepage}

\usepackage{soul} 
\usepackage[normalem]{ulem}
\setulcolor{blue}

\usepackage{stmaryrd} 
\usepackage{preamble} 
\theoremstyle{plain}

\newtheorem{thm}{Theorem}[section]
\newtheorem{lem}[thm]{Lemma}

\theoremstyle{remark}

\newtheorem{rem}[thm]{Remark}
\newtheorem{example}[thm]{Example}

\theoremstyle{definition}
\newtheorem{defn}[thm]{Definition}
\newtheorem{assump}[thm]{Assumption}
\crefname{thm}{Theorem}{Theorems}
\crefname{lem}{Lemma}{Lemmas}

\numberwithin{equation}{section}

\begin{document}

\title{Asymptotics of ultra-high-dimensional generalized spiked sample covariance matrix}
\author[1]{Wonjun Seo\thanks{E-mail: wseo@ucdavis.edu}}
\affil[1]{Department of Statistics, University of California, Davis}
\date{\vspace{-5ex}}
\maketitle
\begin{center}

{\bfseries Abstract}

\end{center}

\begin{adjustwidth}{0.25in}{0.25in}

\fontsize{8pt}{10pt}\selectfont

This paper investigates the asymptotics of eigenstructure of sample covariance matrix under the spiked covariance matrix model in ultra-high-dimensional settings, where the dimensionality can grow much faster than the sample size with $ p \asymp n^{\alpha} $, $ \alpha > 1 $. We establish the first-order convergence limits of eigenvalue locations and eigenvector projections of properly scaled sample covariance matrix. Our results are extensions of \cite{bloemendal16,ding21}.

\end{adjustwidth}



\section{Introduction} \label{sec:intro}
The sample covariance matrix is of fundamental importance in multivariate statistics. In the classical regime, where the dimensionality $ p $ is fixed while the sample size $ n $ tends to infinity, it is well-known that the sample covariance matrix is a consistent estimator of the population covariance matrix. Consequently, the plug-in principle, replacing population covariance matrix by the sample covariance matrix, yields asymptotically optimal procedures for a wide range of methodologies, including principal component analysis, discriminant analysis, and Hotelling's $T^2$ test \cite{anderson03,muirhead82}.

However, this breaks down when $ p $ is no longer negligible compared to $ n $, as the sample covariance matrix loses its consistency \cite{yzb15}. In particular, the eigenstrucure of the sample covariance matrix can deviate significantly from its population courterpart, leading to substantial distortions in statistical procedures that rely on the plug-in principle. The situation becomes even more severe when $ p $ is much larger than $ n $, which is common in modern high dimensional settings, as procedures that rely on the precision matrix may fail completely.

To address this issue, a specific structural assumption is often imposed on the population covariance matrix $ \Sigma $. One of the most widely studied frameworks in this context is the spiked covariance matrix model, first introduced by Johnstone \cite{johnstone01}. Johnstone's spiked covariance matrix model assumes that $ p $ and $ n $ are comparable and that the population covariance matrix is a finite-rank perturbation of the identity matrix as follows
\begin{equation} \label{eq:Johnstone}
	\widetilde{\Sigma}^{\mathrm{Johnstone}} = I_{p} + \sum_{i=1}^{r} \sfd_{i} \bfv_{i} \bfv_{i}^{\top},
\end{equation}
where $ r \in \bbN $ is a fixed number, $ \sfd_{i} > 0 $ are some positive fixed numbers, and $ \bfv_{i} \in \bbR^{p} $ are orthonormal vectors. The key idea of this model is that a few \emph{spikes} in the spectrum represent dominant directions or factors, while the remaining eigenvalues are set to be 1. Over the past few years, many efforts have been made to analyze the asymptotic behavior of the eigenstructure of the sample covariance matrix under this model. Building on these results, a variety of statistical procedures have been developed. We refer to \cite{pa14,jp18,bs10,yzb15,cl22} and references therein for a comprehensive review.

Despite its success, Johnstone’s spiked covariance matrix model has certain limitations. In particular, the assumption that the bulk eigenvalues are all equal to 1 can be overly restrictive in practice, where the underlying covariance structure may be more complex. Moreover, most existing results are established in the regime where $ p $ is comparable to $ n $, and do not directly extend to the regime where $ p $ is much larger than $ n $.

In this paper, we study the asymptotic eigenstructure of sample covariance matrix under a more generalized spiked covariance matrix model and the regime where $ p $ is much larger than $ n $. More specifically, we consider
\begin{equation} \label{eq:sample covmat intro}
	\widetilde{Q} = \widetilde{\Sigma}^{1/2} XX^{\top} \widetilde{\Sigma}^{1/2},
\end{equation}
where $X$ is a $ p \times n $ random matrix whose entries are centered and properly scaled i.i.d\ real random variables, and $ \widetilde{\Sigma} $ is a $ p \times p $ positive definite matrix. Unlike the Johnstone's spiked covariance matrix model in \eqref{eq:Johnstone}, $ \widetilde{\Sigma} $ is assumed to be a finite-rank perturbation of a general positive definite matrix $ \Sigma_{0} $ with eigenvalues $ \{ \sigma_{i} \} $ and eigenvectors $ \{ \bfv_{i} \} $, given by
\begin{equation} \label{eq:spiked model intro}
	\widetilde{\Sigma} = \Sigma_{0} + \sum_{i=1}^{r} \sfd_{i} \sigma_{i} \bfv_{i} \bfv_{i}^{\top},
\end{equation}
under mild conditions. For the dimensionality, we assume that
\begin{equation} \label{eq:uhd intro}
	p \asymp n^{\alpha}, \ \text{ for some constant } \alpha \in (1, \infty),
\end{equation}
so that $ p $ is allowed to grow much faster than $ n $. To distinguish this regime from the comparable regime where $ \alpha = 1 $, we refer to it as the ultra-high-dimensional regime.

\subsection{Related works} \label{ssec:review}
In this subsection, we briefly review existing related results on the spiked covariance matrix model. For the notational convenience, we denote
\begin{equation} \label{eq:phi}
	\phi \equiv \phi_{n} := \frac{p}{n}.
\end{equation}

The first major finding was the identification of the \emph{outliers} in the spectrum of $ \widetilde{Q} $, established by Baik, Ben Arous, and Peché \cite{bbp05}. They showed that if $ \sfd_{i} > \phi^{1/2} $, the $ i $ th largest eigenvalue of $ \widetilde{Q} $ separates from the majority of the other eigenvalues, under the Johnstone's spiked covariance matrix model in \eqref{eq:Johnstone}, for $ \alpha = 1 $ in \eqref{eq:uhd intro}, assuming complex Gaussian entries for $ X $. This phenomenon is now widely known as the BBP transition. The behavior of outlier eigenvectors was established by Paul \cite{paul07}. More precisely, the $ i $ th eigenvector of $ \widetilde{Q} $ is concentrated on a cone with the axis parallel to $ \bfv_{i} $ if $ \sfd_{i} > \phi^{1/2} $, under the same model and regime, assuming real Gaussian entries. Since these foundational works, outlier eigenvalues and eigenvectors of sample covariance matrices have been extensively investigated under weaker assumptions and in various settings \cite{bs06,by08,bn11,by12,bloemendal13,bloemendal16,wf17,ding21,bdww22}.

The behavior of non-outlier eigenvalues and eigenvectors can be analyzed using the local laws. In \cite{bloemendal16}, the authors established the eigenvalue sticking and delocalization, which characterize the asymptotic behavior of non-outlier eigenvalues and eigenvectors, respectively. They also derived the large deviation bounds of the locations of outlier eigenvalues and the generalized components of outlier eigenvectors. These results are based on the isotropic local laws developed in \cite{bloemendal14}, and are proved under Johnstone’s model and for $ \alpha \in (0, \infty) $. In a similar spirit, \cite{ding21} derived the analogous results under the generalized spiked covariance matrix model in \eqref{eq:spiked model intro} for $ \alpha = 1 $, using the anisotropic local laws established by Knowles and Yin \cite{ky17}.

Our paper can be viewed as a synthesis of these two works, as we consider the generalized spiked covariance matrix model with $ \alpha \in (1, \infty) $. While the outlier eigenvalues in this setting have recently been studied in \cite{jing25}, our work also incorporates the behavior of non-outlier eigenvalues and both outlier and non-outlier eigenvector projections.

Finally, we mention that the BBP-type transitions have also been established for other classes of spiked random matrices arising in statistics, including signal-plus-noise matrices \cite{benaych11,benaych12,ding20,bdw21,feldman23}, spiked separable sample covariance matrices \cite{dy21}, spiked Fisher matrices \cite{wy17}, and spiked correlation matrices \cite{mjmy21}. Another important class in random matrix theory is the Wigner matrices, and the analogous results for the deformed Wigner matrices have been studied in \cite{bn11,ky13,prs13,ky14,ls16}.

\subsection{An overview of main results} \label{ssec:overview}

In this subsection, a brief overview of main results is provided.

The first set of results concerns the asymptotic locations of non-trivial eigenvalues $ \{ \widetilde{\lambda}_{i} \}_{i=1}^{n} $ of $ \widetilde{Q} $. Under the assumption that the spike magnitudes $ { \sfd_{i} }_{i=1}^{r} $ exceed certain thresholds given in \eqref{eq:spike1}, the first $ r $ largest eigenvalues $ \{ \widetilde{\lambda}_{i} \}_{i=1}^{r} $ emerge as outliers. They detach from the rest eigenvalues and concentrate around deterministic locations determined by $ \{\sfd_{i}\}_{i=1}^{r} $. On the other hand, the non-outlier eigenvalues $ \{ \widetilde{\lambda}_{i} \}_{i=r+1}^{n} $ concentrate around the corresponding quantiles of the density $ \varrho $, which is the limiting spectral distribution of the sample covariance matrix under the non-spiked model; see Section \ref{ssec:asymp law} for details.

Next, the results on the eigenvectors $ \{ \widetilde{\bfu}_{i} \}_{i=1}^{n} $ of $ \widetilde{Q} $ are summarized. To avoid the ambiguity of the phase of $ \widetilde{\bfu}_{i} $ and to quantify the alignment with the population eigenvectors $ \{ \bfv_{j} \} $, the squared spectral projections $ \langle \widetilde{\bfu}_{i}, \bfv_{j} \rangle^{2} $ are considered. The first part of Theorem \ref{thm:eigenvector} shows that each outlier eigenvector $ \widetilde{\bfu}_{i} $, $ 1 \le i \le r $, concentrates on a cone with the axis parallel to corresponding population eigenvector $ \bfv_{i} $, with a deterministic aperture depending on $ \sfd_{i} $. On the other hand, the non-outlier eigenvectors are delocalized in any direction of $ \bfv_{j} $. 

Although Theorem \ref{thm:eigenvector} deals with individual eigenvector projections, weighted sums of projections are often of interest in statistical applications, such as principal component analysis (cf. Example \ref{ex:pca}) and optimal shrinkage estimation \cite{donoho18,dly24}. In such applications, one is interested in the deterministic equivalent of 
\[ \sum_{j=1}^{p} l_{j} \langle \widetilde{\bfu}_{i}, \bfv_{j} \rangle^{2}, \]
where $ \{ l_{j} \} $ is some positive deterministic sequence. For $ 1 \le i \le r $, Theorem 3.4 \eqref{eq:outlier eigenvector} alone is insufficient for this purpose, as it only yields the bound $ | \langle \widetilde{\bfu}_{i}, \bfv_{j} \rangle |^{2} \le p^{-1+\epsilon} $ with high probability for $ r+1 \le j \le p $. To address this issue, a perturbation argument introduced in \cite{dly24} is employed, leading to the first-order limit of the weighted projections in Theorem \ref{thm:spiked outlier eigenvector projection sum}.

Before proceeding, we discuss several possible future works. First, we assume the spikes are above the critical points by constant-order gap for simplicity. This means that we only consider the somewhat trivial supercritical regime. We believe that the phase transition occurs on the scale $ n^{-1/3} $, as in other models. Second, the second-order limits (distributional limits) are also important to develop statistical inference procedures. We refer to \cite{bloemendal16,bdww22} for those who are interested in joint distribution of outlier eigenvalues and eigenvectors under simpler settings. Lastly, the deterministic equivalent of $ \sum_{j=1}^{p} l_{j} | \langle \widetilde{\bfu}_{i}, \bfv_{j} \rangle |^{2} $ for $ r+1 \le i \le p $ is also important in optimal shrinkage estimation. However, we do not pursue this direction, as it requires establishing the quantum unique ergodicity (QUE) estimate, which involves more sophisticated techniques. The QUE estimates for sample covariance matrix under the comparable regime $ \alpha = 1 $ have been recently established in \cite{dly24}. We believe similar argument with a slight modification will also work in the ultra-high-dimensional regime. We will pursue these generalizations in our future works.

\subsubsection*{Organzation}
This paper is organized as follows. In Section \ref{sec:model}, the non-spiked and spiked covariance matrix model, along with key assumptions, are introduced, and some preliminary results on the asymptotic laws of sample covariance matrices under the non-spiked covariance matrix model are summarized. Section \ref{sec:results} presents the main results of the paper, where the eigenvalue locations and eigenvector projections under the spiked covariance matrix model are derived. In Appendix \ref{app:local law}, we introduce the local laws, which serve as key inputs to the proofs of the main results. The proofs of the results are deferred to Appendix \ref{app:proof}.

\subsubsection*{Notation}
In this paper, we follow the notational conventions commonly used in the random matrix theory and high-dimensional statistics literature. The sample size $ n $ is the fundamental parameter, and we often omit the explicit dependence on $ n $. For instance, although \eqref{eq:uhd intro} indicates that the dimensionality $ p $ depends on $ n $, we write $ p \equiv p_{n} $ as $ p $. A constant that does not depend on $ n $ is referred to as a \textit{fixed} constant. Let $ a_{n} $ and $ b_{n} $ be positive real sequences. We write $ a_{n} \asymp b_{n} $ if there exists a positive fixed constant $ C > 0 $ such that $ C^{-1} a_{n} \le b_{n} \le C a_{n} $ and $ a_{n} \gtrsim b_{n} $ if $ a_{n} \ge C b_{n} $ for sufficiently large $ n $. In addition, we write $ a_{n} \gg b_{n} $ if $ a_{n}/b_{n} \to \infty $ as $ n \to \infty $. The Euclidean norm of a vector $ \mathbf{v} \in \mathbb{R}^{p} $ is denoted by $ \| \bfv \| $. The operator norm of a matrix $ A \in \mathbb{R}^{p \times p} $ is denoted by $ \| A \| $. The complex number will be written as $ z = E + \i \eta $, where $ E $ is the real part and $ \eta $ is the imaginary part. We let $ \bbC_{+} = \left \{ z = E + \mathrm{i} \eta: \eta > 0 \right \} $. For any $ i, j \in \bbN $ with $ i < j $, we denote $ \llb i \rrb = \{ 1, 2, \dots, i \} $ and $ \llb i, j \rrb = \{ i, i+1, \dots, j \} $.

\section{Model} \label{sec:model}
In this section, we introduce our model, assumptions, and several notations required to state the main results.

\subsection{Spiked covariance matrix model and assumptions} \label{ssec:spiked}
Throughout the paper, we consider the spiked covariance matrix model with general structures, as considered in \cite{ding21, dly24}.
Let $ \Sigma_{0} $ be a deterministic $ p \times p $ positive definite matrix admitting the spectral decomposition
\begin{equation} \label{eq:non-spiked model}
	\Sigma_0 = V \Lambda_{0} V^{\top} = \sum_{i=1}^{p} \sigma_{i} \bfv_{i} \bfv_{i}^{\top}, \quad 0 < \sigma_{p} \le \dots \le \sigma_{1},
\end{equation}
where $\{\sigma_i\}$ are the eigenvalues and $\{\bfv_i\}$ are the associated eigenvectors of $ \Sigma_{0} $. For some fixed integer $r \in \bbN$, we introduce $ r $ \emph{spikes} to $ \Sigma_0 $, resulting in another $ p \times p $ positive definite matrix $ \widetilde{\Sigma} $ with spectral decomposition
\begin{equation} \label{eq:spiked model}
	\widetilde{\Sigma}= V \widetilde{\Lambda} V^{\top} = \sum_{i=1}^{p} \widetilde{\sigma}_{i} \bfv_{i} \bfv_{i}^{\top}, \quad 0 < \widetilde{\sigma}_{p} \le \dots \le \widetilde{\sigma}_{1},
\end{equation}
where the spiked eigenvalues $ \{ \widetilde{\sigma}_{i} \} $ are defined as
\begin{equation} \label{eq:spiked eigenvalue}
	\widetilde{\sigma}_{i} = \begin{cases}
		(1 + \sfd_{i})\sigma_{i}, & i \in \llb r \rrb \\
		\sigma_{i}, & i \in \llb r+1, p \rrb 
	\end{cases},
\end{equation}
where $ 0 < \sfd_r \le \dots \le \sfd_{1} $ are parameters used to characterize the spikes.
Note that when $\sfd_i \equiv 0$ for all $1 \leq i \leq r$ or $r=0$, $ \widetilde{\Sigma} $ contains no spikes and coincides with $\Sigma_0$. In what follows, we refer to $ \Sigma = \Sigma_0 $ as the \emph{non-spiked covariance matrix model} and $ \Sigma = \widetilde{\Sigma} $ as the \emph{spiked covariance matrix model}.

Under each model, we define the corresponding sample covariance matrix as 
\begin{equation} \label{eq:sample covmat}
	Q_{0} = \Sigma_{0}^{1/2} XX^{\top} \Sigma_{0}^{1/2}, \quad \widetilde{Q} = \widetilde{\Sigma}^{1/2} XX^{\top} \widetilde{\Sigma}^{1/2},
\end{equation}
where $ X \in \bbR^{p \times n} $ is a random matrix whose entries are i.i.d.\:centered random variables. By using the same realization of $ X $ for both $ Q_{0} $ and $ \widetilde{Q} $, any discrepancy between these two sample covariance matrices originates solely from the difference between the population covariance matrices $ \Sigma_{0} $ and $ \widetilde{\Sigma} $.

We impose the following assumptions on $ X $, $ \Sigma_{0} $, and $ \widetilde{\Sigma} $.

\begin{assump} ~\ \label{assump:main}
	\begin{enumerate}[label = (\roman*)]
		\item \textbf{Dimensionality.} There exists a constant $ \alpha \in (1, \infty) $ such that
		\begin{equation} \label{eq:uhd}
			p \asymp n^{\alpha}.
		\end{equation}
		
		\item \textbf{Entries of $ X $.} Let $ X = (x_{ij}) $ be a $ p \times n $ matrix whose entries are i.i.d.\:real random variables with
		\begin{equation} \label{eq:centered and scaling}
			\bbE x_{ij} = 0, \quad \bbE x_{ij}^{2} = \frac{1}{\sqrt{pn}}.
		\end{equation}
		We further assume that for each $ k \in \bbN $, $ k $-th norm of $ (pn)^{1/4} x_{ij} $ is uniformly bounded, i.e., there exists a constant $ C_{k} > 0 $ such that
		\begin{equation} \label{eq:moments}
			\bbE \left | (pn)^{1/4} x_{ij} \right |^{k} \le C_{k}.
		\end{equation}
		
		\item \textbf{Spectrum of $ \Sigma_{0} $.} Let $ \varsigma \in (0, 1) $ be a fixed small constant. We assume that $ \{ \sigma_{i} \} $ in \eqref{eq:non-spiked model} satisfy
		\begin{equation} \label{eq:spectrum of Sigma0}
			\varsigma \le \sigma_{p} \le \cdots \le \sigma_{1} \le \varsigma^{-1}.
		\end{equation}
		
		\item \textbf{Spikes.} Let $ \varpi \in (0, 1) $ be a fixed small constant. We assume that $ \{ \widetilde{\sigma}_{i} \}_{i=1}^{r} $ in \eqref{eq:spiked eigenvalue} satisfy that for all $ i \in \llb r \rrb $
		\begin{equation} \label{eq:spike1}
			-\phi^{1/2} (\mathsf{c}_{1} + \varpi)^{-1} < \widetilde{\sigma}_{i} < \phi^{1/2} \varpi^{-1},
		\end{equation}
		where $ \phi $ is defined in \eqref{eq:phi}, and $ \mathsf{c}_{1} < 0 $ with $ |\mathsf{c}_{1}| \asymp 1 $ will be defined in Lemma \ref{lem:properties of rho and f} (i) below.
		We further assume that 
		\begin{equation} \label{eq:spike2}
			\min_{i \neq j \in \llb r \rrb} \left | \widetilde{\sigma}_{i} - \widetilde{\sigma}_{j} \right | > \phi^{1/2} \varpi.
		\end{equation}
	\end{enumerate}
\end{assump}

\begin{rem}
	Several remarks are in order. First, condition (i) indicates that we consider the ultra-high-dimensional regime, in which $ p $ grows much faster than $ n $. Second, the scaling $ (pn)^{-1/2} $ in \eqref{eq:centered and scaling} is adopted to accommodate the ultra-high-dimensional regime. In contrast, in the classical low-dimensional regime $ (\alpha < 1) $ and the comparable regime $ (\alpha = 1) $, the scaling $ n^{-1} $ is typically used. The moment assumption in \eqref{eq:moments} can be relaxed to finitely many moments by a truncation argument, but such generalizations are not pursued here for simplicity. Third, \eqref{eq:spectrum of Sigma0} is a mild condition ensuring that the eigenvalues of $ \Sigma_{0} $ are bounded from both above and below. Finally, the lower bound in condition \eqref{eq:spike1} is imposed to ensure that $\sfd_i$ represents a true spike in the sense that the associated eigenvalue of the sample covariance matrix $ \widetilde{Q} $ emerges as an outlier. For simplicity, we do not consider the case where $ \widetilde{\sigma}_{i} $ is smaller than $ -\phi^{1/2} \mathsf{c}_{1}^{-1} $, in which case no outlier is created. Although this transition is not analyzed explicitly, the quantity $ -\phi^{1/2} \mathsf{c}_{1}^{-1} $ serves as the critical threshold distinguishing the subcritical and supercritical regimes, as in other BBP-type transitions. Moreover, we assume a constant-order gap between $ \phi^{-1/2} \widetilde{\sigma}_{i} $ and $ -\sfc_{1}^{-1}$ for clarity of exposition. This condition can be relaxed to gaps of order $ n^{-1/3 + \epsilon} $ for any small constant $\epsilon > 0$, following \cite{bloemendal16,dy21}. The upper bound in \eqref{eq:spike1} is imposed only for notational convenience and can be readily removed without difficulty. Condition \eqref{eq:spike2} ensures that the spiked eigenvalues $\{\widetilde{\sigma}_{i}\}_{i=1}^{r}$ are well-separated, so that the corresponding outlier eigenvectors of $ \widetilde{Q} $ are uniquely defined. Nevertheless, this condition can be relaxed to include degenerate cases by considering the eigenspace as in \cite{dy21, bloemendal16}. We will explore these generalizations in future work.
\end{rem}

\subsection{Asymptotic laws under the non-spiked covariance matrix model} \label{ssec:asymp law}
In this subsection, we summarize some results on the asymptotic laws of the eigenvalues of sample covariance matrix $ Q_{0} $ under the non-spiked covariance matrix model. Recall that for any $ p \times p $ symmetric matrix $ A $, its empirical spectral distribution (ESD) is defined as the empirical probability measure $ \mu_{A} = \frac{1}{p} \sum_{i=1}^{p} \delta_{\lambda_{i}(A)} $, where $ \{ \lambda_{i}(A) \} $ are the eigenvalues of $ A $. For a probability measure $ \mu $ on $ \bbR $, its Stieltjes transform $ m_{\mu}: \mathbb{C}^{+} \to \mathbb{C}^{+} $ is defined by 
\[ m_{\mu}(z) = \int_{\bbR} \frac{1}{x-z} \mu (\dif x), \quad z \in \mathbb{C}^{+}. \]

We now characterize the asymptotics of the ESD of $ Q_{0} $ under Assumption \ref{assump:main} (i)-(iii). Since $ Q_{0} $ and its companion $ \cQ_{0} := X^{\top} \Sigma_{0} X $ share the non-trivial eigenvalues, we focus on the ESD of $ \cQ_{0} $. From \cite{ding23}, it is known that the ESD of $ \cQ_{0} $ has the asymptotic deterministic equivalent, denoted by $ \varrho \equiv \varrho_{\Sigma_{0}} $. The probability measure $ \varrho $ can be best described through its Stieltjes transform $ m \equiv m_{\Sigma_{0}} $. To be more specific, as given in \cite[Lemma 2.3]{ding23}, for $ z \in \bbC^{+} $, $ m(z) $ is characterized as the unique solution to the equation
\begin{equation} \label{eq:self-consistent equation}
	z = f(m(z)), \quad m(z) \in \bbC^{+},
\end{equation}
where, recalling that $ \phi = p/n $ in \eqref{eq:phi}, $ f \equiv f_{\Sigma_{0}}: \bbC^{+} \to \bbC^{+} $ is defined as
\begin{equation} \label{eq:f}
	f(x) = - \frac{1}{x} + \frac{\phi^{1/2}}{p} \sum_{i=1}^{p} \frac{\sigma_{i}}{1+\phi^{-1/2} \sigma_{i} x}.
\end{equation}
Note that the domain of $ f $ can be readily extended to the real projective line $ \overline{\mathbb{R}} = \mathbb{R} \cup \{\infty\} $ except $ 0 $ and $ \{ \phi^{1/2} \sigma_{i}^{-1} \} $.

The following lemma established in \cite{ding24} summarizes some properties of $ f $ and $ \varrho $. Denote
\begin{equation} \label{eq:m_1}
	\mathfrak{m}_{1} := \int x \pi_{0}(\dif x),
\end{equation}
where $ \pi_{0} := \frac{1}{p} \sum_{i=1}^{p} \delta_{\sigma_{i}} $ is the ESD of $ \Sigma_{0} $.

\begin{lem} \label{lem:properties of rho and f}
	Under Assumption \ref{assump:main} (i)-(iii) holds, the following hold for sufficiently large $ n $.
	\begin{enumerate}[label = (\arabic*)]
		\item $ f $ has two critical points, denoted by $ \mathsf{c}_{1} < 0 $ and $ \mathsf{c}_{2} > 0 $, with $ | \sfc_{k} | \asymp 1 $.
		\item Define $ \gamma_{+} := f(\mathsf{c}_{1}) $ and $ \gamma_{-} := f(\mathsf{c}_{2}) $. Then, we have
		\begin{equation} \label{eq:properties of gamma pm}
			\gamma_{+} \ge \gamma_{-} > 0, \quad |\gamma_{\pm} - \phi^{1/2} \mathfrak{m}_{1}| = \mathrm{O}(1),
		\end{equation}
		and
		\begin{equation} \label{eq:support of rho}
			\operatorname{supp} \varrho \cap (0, \infty) = [ \gamma_{-}, \gamma_{+} ].
		\end{equation}
		\item $ \varrho $ has the square-root behavior near the edge $ \gamma_{\pm} $.
	\end{enumerate}
\end{lem}

\begin{proof}
	See Lemma 2.5 and Section B.2 in \cite{ding24}.
\end{proof}

\begin{rem} \label{rem:support of rho}
	From \eqref{eq:self-consistent equation}-\eqref{eq:f} and Lemma \ref{lem:properties of rho and f}, the asymptotics of the ESD $ Q_{0} $ can be obtained once we know $ \Sigma_{0} $, more precisely its ESD $ \pi_{0} $. In particular, the support of $ \varrho $ is given by an bounded interval containing $ \phi^{1/2} \mathfrak{m}_{1} $.
\end{rem}

\begin{rem} \label{rem:Johnstone}
	When $ \Sigma_{0} = I_{p} $, that is, when all $ \{ \sigma_{i} \} $ are equal to one, \eqref{eq:f} feduces to
	\begin{equation} \label{eq:f Johnstone}
		f^{\mathrm{MP}}(x) = - \frac{1}{x}  + \frac{\phi^{1/2}}{1 + \phi^{-1/2} x}.
	\end{equation}
	The corresponding asymptotic density, denoted by $ \varrho^{\mathrm{MP}} $, is the celebrated Marchenko-Pastur law \cite{mp67}, given by
	\begin{equation}
		\varrho^{\mathrm{MP}}(\dif x) := \frac{\phi^{1/2}}{2\pi} \frac{\sqrt{\left [ (x - \gamma_{-}^{\mathrm{MP}}) ( \gamma_{+}^{\mathrm{MP}} - x ) \right ]_{+}}}{x} \dif x,
	\end{equation}
	where $ \gamma_{\pm}^{\mathrm{MP}} = \phi^{1/2} + \phi^{-1/2} \pm 2 $. See \cite[Remark 2.4]{ding23} and \cite[Section 2.1]{bloemendal14} for details. Note that the critical points of $ f^{\mathrm{MP}} $ are given by 
	\[ \mathsf{c}_{1}^{\mathrm{MP}} = - (1 + \phi^{-1/2})^{-1} \quad \text{and} \quad \mathsf{c}_{2}^{\mathrm{MP}} = (1 - \phi^{-1/2})^{-1}. \]
\end{rem}

\begin{rem}
	Under the comparable regime $ (\alpha = 1) $, the function $ f $ may have an even number of critical points, with convention that a degenerate critical point is counted twice, and that the support of $ \varrho $ can have multiple bulk components whose edges are the value of $ f $ at each critical point. See \cite[Remark 2.6]{ding23} and \cite[Lemma 2.4-2.6]{ky17} for details.
\end{rem}

\section{Main results} \label{sec:results}
In this section, we present our main results. For a precise description, we use the notion of \emph{stochastic domination}, first introduced in \cite{eky13}. In random matrix theory literature, this has become a standard tool for presenting the statements of the form ``$ \sfX $ is bounded with high probability by $ \sfY $ up to a small power of $ n $".
\begin{defn}[Stochastic domination] \label{defn:stochastic domination}
	Let
	\[ \sfX = \left ( \sfX^{(n)}(u): n \in \bbN, u \in \sfU^{(n)} \right ), \quad \sfY = \left ( \sfY^{n}(u): n \in \bbN, u \in \sfU^{(n)} \right ) \]
	be two families of nonnegative random variables, where $ \sfU^{(n)} $ is a possibly $ n $-dependent parameter set. $ \sfX $ is said to be stochastically dominated by $ \sfY $, uniformly in $ u $, if for any small $ \epsilon > 0 $ and large $ D > 0 $, we have
	\[ \sup_{u \in \sfU^{(n)}} \bbP \left ( \sfX^{(n)}(u) > n^{\epsilon} \sfY^{(n)}(u) \right ) \le n^{-D}, \]
	for all large enough $ n \ge N_{0}(\epsilon, D) $, and is denoted by $ \sfX \prec \sfY $. For a complex family $ \sfX $, we write $ \sfX \prec \sfY $ if $ |\sfX| \prec \sfY $. For two deterministic nonnegative families $ \sfX $ and $ \sfY $, we write $ \sfX \prec \sfY $ if $ \sfX \le n^{\epsilon} \sfY $ for any small $ \epsilon > 0 $ whenever $ n $ is sufficiently large. We also write $ \sfX = \mathrm{O}_{\prec}(\sfY) $ if $ \sfX \prec \sfY $.
\end{defn}

We write the spectral decomposition of $ \widetilde{Q} $ as
\[ \widetilde{Q} = \sum_{i=1}^{p} \widetilde{\lambda}_{i} \widetilde{\bfu}_{i} \widetilde{\bfu}_{i}^{\top}, \]
where $ \widetilde{\lambda}_{1} \ge \dots \ge \widetilde{\lambda}_{p} \ge 0 $ and $ \{ \widetilde{\bfu}_{i} \} $ are the eigenvalues and the corresponding eigenvectors.

\subsection{Eigenvalue locations} \label{ssec:eigenvalues}
We start with the eigenvalues of $ \widetilde{Q} $. We introduce some notations first. For $ i \in \llb r \rrb $, let 
\begin{equation} \label{eq:a_i}
	\mathfrak{a}_{i} := f \left ( - \phi^{1/2} \widetilde{\sigma}_{i}^{-1} \right ),
\end{equation}
where $ f $ is defined in \eqref{eq:f}. Recall the asymptotic law $ \varrho $ introduced in Section \ref{ssec:asymp law}. We define its quantiles $ \{ \gamma_{i} \} $ by
\begin{equation}\label{eq:gamma_i}
	\int_{\gamma_{i}}^{\infty} \varrho(\dif x) = \frac{i}{n}, \quad i \in \llb n \rrb.
\end{equation}

\begin{thm}
	\label{thm:eigenvalue}
	Suppose that Assumption \ref{assump:main} holds.
	\begin{enumerate}[label = (\roman*)]
		\item For $ i \in \llb r \rrb $ uniformly, we have 
		\begin{equation} \label{eq:outlier eigenvalue}
			\left | \widetilde{\lambda}_{i} - \mathfrak{a}_{i} \right | \prec n^{-1/2}.
		\end{equation}
		\item For $ i \in \llb r+1, n \rrb $ uniformly, we have
		\begin{equation} \label{eq:non-outlier eigenvalue}
			\left | \widetilde{\lambda}_{i} - \gamma_{i-r} \right | \prec ( \min \{ i, n+1-i \} )^{-1/3} n^{-2/3}.
		\end{equation}
	\end{enumerate}
\end{thm}

Theorem \ref{thm:eigenvalue} describes the asymptotic behavior of the non-trivial eigenvalues of $ \widetilde{Q} $. More precisely, it identifies deterministic equivalents for their locations together with precise error bounds. These deterministic equivalents depend on whether the eigenvalues correspond to the spiked eigenvalues or to the remaining non-spiked eigenvalues.

For the first $ r $ eigenvalues, the deterministic equivalents are given by $ \{ \mathfrak{a}_{i} \}_{i=1}^{r} $. Under Assumption \ref{assump:main} (iv), $ \{ \mathfrak{a}_{i} \}_{i=1}^{r} $ lie outside the support of $ \varrho $ of $ Q_{0} $. In particular, they are greater than $ \gamma_{+} $ by constant-order gap; see \eqref{eq:a_i separate from support} below. Thus, \eqref{eq:outlier eigenvalue} shows that the first $ r $ eigenvalues $ \{ \widetilde{\lambda}_{i} \}_{i=1}^{r} $ emerge as outliers and concentrate around $ \{ \mathfrak{a}_{i} \}_{i=1}^{r} $ with error of order $ n^{-1/2} $.

On the other hand, the deterministic equivalents of the remaining eigenvalues are given by the quantiles of $ \varrho $. Specifically, \eqref{eq:non-outlier eigenvalue} shows that $ \widetilde{\lambda}_{i} $ is close to $ \gamma_{i-r} $ for $ i \in \llb r+1,n \rrb $. The convergence rate depends on the location of the quantile. It is of order $ n^{-2/3} $ near the edges and of order $ n^{-1} $ in the bulk. These rates match the typical spacing of the quantiles, which follows from the square-root behavior of $ \varrho $ at the edges, as stated in Lemma \ref{lem:properties of rho and f} (3).

\begin{rem}
	In the proof of \eqref{eq:non-outlier eigenvalue} below, we combine two estimates: the \emph{eigenvalue sticking} and the \emph{rigidity}. Heuristically, the eigenvalue sticking can be expressed as $ \widetilde{\lambda}_{i} \approx \lambda_{i-r} $, where $ \{\lambda_{i}\} $ are the eigenvalues of $ Q_{0} $, meaning that the non-outlier eigenvalues of $ \widetilde{Q} $ ``stick'' to the corresponding eigenvalues of $ Q_{0} $. On the other hand, the rigidity implies that each $ \lambda_{i} $ concentrates around the quantile $ \gamma_{i} $ of $ \varrho $. See \eqref{eq:eigenvalue sticking} and Lemma \ref{lem:rigidity} below for the precise statements.
\end{rem}

\begin{example} \label{ex:Johnstone eigenvalue}
	Consider Johnstone's spiked covariance matrix model, that is, $ \Sigma_{0} = I_{p} $ as in Remark \ref{rem:Johnstone}. Then, the condition \eqref{eq:spike1} in Assumption \ref{assump:main} is satisfied whenever $ \sfd_{i} > \phi^{1/2} (1 + \varpi) $. Under this case, \eqref{eq:a_i} with $ f^{\mathrm{MP}} $ in \eqref{eq:f Johnstone} gives
	\[ \mathfrak{a}_{i}^{\mathrm{Johnstone}} = \phi^{1/2} + \phi^{-1/2} + \phi^{-1/2} \sfd_{i} + \phi^{1/2} \sfd_{i}^{-1}, \]
	which recovers Theorem 2.3 in \cite{bloemendal16}.
\end{example}

\subsection{Eigenvector projections} \label{ssec:eigenvectors}
In this subsection, we present the results on the eigenvector projections. For $ i \in \llb r \rrb $, denote 
\begin{equation} \label{eq:b_i}
	\mathfrak{b}_{i} = \left ( \phi^{1/2} \widetilde{\sigma}_{i}^{-1} \right ) \frac{f' \left ( - \phi^{1/2} \widetilde{\sigma}_{i}^{-1} \right )}{f \left ( - \phi^{1/2} \widetilde{\sigma}_{i}^{-1} \right )}.
\end{equation}
Recall that $ \{ \bfv_{j} \} $ are the eigenvectors of population covariance matrix $ \widetilde{\Sigma} $.

\begin{thm} \label{thm:eigenvector}
	Suppose that Assumption \ref{assump:main} holds.  Then,
	\begin{enumerate}[label=(\roman*)]
		\item For $ i \in \llb r \rrb $ uniformly, we have
		\begin{equation} \label{eq:outlier eigenvector}
			\langle \widetilde{\bfu}_{i}, \bfv_{j} \rangle^{2} = \begin{cases}
				\delta_{ij} \mathfrak{b}_{i} + \mathrm{O}_{\prec}(p^{-1/2}), \quad &j \in \llb r \rrb, \\
				\mathrm{O}_{\prec}(p^{-1}), \quad &j \in \llb r+1, p \rrb.
			\end{cases}
		\end{equation}
		\item For $ i \in \llb r+1, n \rrb $ uniformly, we have
		\begin{equation} \label{eq:non-outlier delocalization}
			\langle \widetilde{\bfu}_{i}, \bfv_{j} \rangle^{2} \prec \begin{cases}
				(pn)^{-1/2}, \quad j \in \llb r \rrb, \\
				p^{-1}, \quad j \in \llb r+1, p \rrb
			\end{cases}
		\end{equation}
	\end{enumerate}
\end{thm}

Theorem \ref{thm:eigenvector} shows that the first-order deterministic equivalents for individual eigenvector projections are available only for the outlier eigenvectors in their corresponding spike directions. More precisely, for $ i \in \llb r \rrb $, the projection $ \langle \widetilde{\bfu}_{i}, \bfv_{j} \rangle^{2} $ has the deterministic equivalent $ \mathfrak{b}_{i} $, whose order is given by $ \phi^{-1/2} $, only when $ j=i $. This means that each outlier eigenvector $ \widetilde{\bfu}_{i} $ is asymptotically localized around the spike direction $ \bfv_i $, in the sense that it concentrate on a cone with the axis parallel to $ \bfv_{i} $ with an aperture given by $ \mathfrak{b}_{i} $. For all other projections, we only obtain high-probability error bounds instead. In particular, non-outlier eigenvectors are delocalized in any direction of $ \{ \bfv_{j} \} $. The theorem also shows that the order of the bound depends on the index $ j $, according to whether $ \bfv_j $ is a spike direction or a non-spike direction.

\begin{example} \label{ex:Johnstone eigenvector}
	Similar to Example \ref{ex:Johnstone eigenvalue}, Theorem \ref{thm:eigenvector} recovers the results for Johnstone's spiked covariance matrix model given in Example 2.13 in \cite{bloemendal16} when $ \Sigma_{0} = I_{p} $. We omit the details.
\end{example}

Finally, we consider weighted sums of the projections of the non-spike directions $ \{ \bfv_{j} \}_{j=r+1}^{p} $ onto an outlier eigenvector
$ \widetilde{\bfu}_{i} $, where $ i \in \llb r \rrb $. Such quantities arise naturally in statistical applications; see Example \ref{ex:pca} below for one such example. The individual projection estimate in Theorem \ref{thm:eigenvector} is not sufficient for this purpose. Indeed, \eqref{eq:outlier eigenvector} only gives
$ | \langle \widetilde{\bfu}_{i}, \bfv_{j} \rangle |^{2} \prec p^{-1} $
for each $ j \in \llb r+1,p \rrb $. Therefore, summing this bound over
$ j $ produces an error of order one, and does not yield a deterministic equivalent for quantities of the form $ \sum_{j=r+1}^{p} \ell_{j} \langle \widetilde{\bfu}_{i}, \bfv_{j} \rangle^{2} $, where $ \{ \ell_{j} \} $ is some deterministic sequence. The following theorem provides such a deterministic equivalent directly at the level of the weighted sum. For $ i \in \llb r \rrb $ and a deterministic positive sequence $ \{ \ell_{j} \}_{j=1}^{p} $, define
\begin{equation} \label{eq:mdot}
	\dot{\sfm}_{0, \fa_{i}}(\mathfrak{a}_{i}) \equiv \dot{\sfm}_{0, \fa_{i}}(\mathfrak{a}_{i}; \ell) := \frac{\phi m'(\mathfrak{a}_{i})}{p \mathfrak{a}_{i}} \sum_{j=1}^{p} \frac{\ell_{j}/\sigma_{j}}{(\phi^{-1/2}m(\mathfrak{a}_{i}) + \sigma_{j}^{-1})^{2}},
\end{equation}
where $ m'(z) = \frac{\dif}{\dif z} m(z) $.

\begin{thm} \label{thm:spiked outlier eigenvector projection sum}
	Suppose that Assumption \ref{assump:main} holds. Let $ \{ \ell_{j} \}_{j=r+1}^{p} $ be a deterministic positive sequence such that $ \tau \le \ell_{j} \le \tau^{-1} $ for some small constant $ \tau $ and $ \ell_{1} = \dots = \ell_{r} = \ell_{r+1} $. Then, we have that
	\begin{equation}
		\left | \sum_{j=r+1}^{p} \ell_{j} \langle \widetilde{\bfu}_{i}, \bfv_{j} \rangle^{2} - \phi^{-1} \fa_{i} \fb_{i} \widetilde{\sigma}_{i} \dot{\sfm}_{0, \fa_{i}}(\fa_{i}) \right | \prec n^{-1/4}.
	\end{equation}
\end{thm}

\begin{rem}
	For $ i \in \llb r \rrb $, the distinction between the case $ j \in \llb r \rrb $ and the case $ j \in \llb r+1, p \rrb $ for $ \langle \widetilde{\bfu}_{i}, \bfv_{j} \rangle^{2} $ is natural in the spiked model. The direction $ \bfv_{i} $ corresponds to a spike, and therefore plays a distinguished role in the behavior of the associated outlier eigenvector $ \widetilde{\bfu}_{i} $. Hence, the projection of $ \widetilde{\bfu}_{i} $ onto $ \bfv_{i} $ is individually tractable. In contrast, the directions
$ \{ \bfv_{j} \}_{j=r+1}^{p} $ are non-spiked directions. Their individual contributions are too small to be resolved separately. Nevertheless, their aggregate contribution admits a tractable deterministic equivalent, as described above.
\end{rem}

\begin{example}[Principal component analysis] \label{ex:pca}
	In this example, we illustrate how the preceding results can be used to analyze a quantity that arises in principal component analysis. For $ i \in \llb r \rrb $, $ \widetilde{\bfu}_{i}^{\top} \widetilde{\Sigma} \widetilde{\bfu}_{i} $ measures the variance explained in the direction of the $ i $th principal component. We use Theorems \ref{thm:eigenvector} and \ref{thm:spiked outlier eigenvector projection sum} to derive its deterministic equivalent.
	\begin{align*}
		\widetilde{\bfu}_{i}^{\top} \widetilde{\Sigma} \widetilde{\bfu}_{i} &= \sum_{j=1}^{p} \widetilde{\sigma}_{j} \langle \widetilde{\bfu}_{i}, \bfv_{j} \rangle^{2} \\
		&= \sum_{j=1}^{r} \widetilde{\sigma}_{j} \langle \widetilde{\bfu}_{i}, \bfv_{j} \rangle^{2} + \sum_{j=r+1}^{p} \widetilde{\sigma}_{j} \langle \widetilde{\bfu}_{i}, \bfv_{j} \rangle^{2} \\
		&= \widetilde{\sigma}_{i} \mathfrak{b}_{i} + \mathrm{O}_{\prec}(\phi^{1/2} p^{-1/2}) + \phi^{-1} \fa_{i} \fb_{i} \widetilde{\sigma}_{i} \dot{\sfm}_{0, \fa_{i}}(\fa_{i}) + \mathrm{O}_{\prec} (n^{-1/4}) \\
		&= \widetilde{\sigma}_{i} \mathfrak{b}_{i} + \phi^{-1} \fa_{i} \fb_{i} \widetilde{\sigma}_{i} \dot{\sfm}_{0, \fa_{i}}(\fa_{i}) + \mathrm{O}_{\prec} (n^{-1/4}),
	\end{align*}
	where
	\[ \dot{\sfm}_{0, \fa_{i}}(\mathfrak{a}_{i}) = \frac{\phi m'(\mathfrak{a}_{i})}{p \mathfrak{a}_{i}} \sum_{j=1}^{p} \frac{\sigma_{j}/\sigma_{j}}{(\phi^{-1/2}m(\mathfrak{a}_{i}) + \sigma_{j}^{-1})^{2}} = \frac{\phi m'(\mathfrak{a}_{i})}{p \mathfrak{a}_{i}} \sum_{j=1}^{p} \frac{1}{(\phi^{-1/2}m(\mathfrak{a}_{i}) + \sigma_{j}^{-1})^{2}}. \]
\end{example}


\appendix

\section{Local laws} \label{app:local law}
In this section, we introduce the local laws, which will be the essential tools for deriving all the theoretical results in this paper. We start with the basic estimates of $ m $. Recall the notations $ \gamma_{\pm} $ and $ \mathfrak{m}_{1} $ defined in Lemma \ref{lem:properties of rho and f} and \eqref{eq:m_1}, respectively. For $ z = E + \i \eta \in \bbC $, let
\begin{equation} \label{eq:kappa}
	\kappa \equiv \kappa(z) := \operatorname{min} \{ |E - \gamma_{+}|, |E - \gamma_{-}| \}.
\end{equation}

\begin{lem}[Basic estimates of $ m $] \label{lem:basic estimates of m}
	Suppose that Assumption \ref{assump:main} (i)-(iii) holds. Let $ \tau \in (0, 1) $ be a fixed small constant. Then, for $ z \in \left \{E+ \i \eta: |E - \phi^{1/2} \mathfrak{m}_{1} | \le \tau^{-1}, \ 0 < \eta \le \tau^{-1} \right \} $,
	\begin{equation} \label{eq:estimate of m}
		|m(z)| \asymp 1
	\end{equation}
	and
	\begin{equation} \label{eq:estimate of im m}
		\im m(z) \asymp \begin{cases}
			\sqrt{\kappa + \eta}, & E \in [\gamma_{-}, \gamma_{+}], \\
			\frac{\eta}{\sqrt{\kappa + \eta}}, & E \notin [\gamma_{-}, \gamma_{+}].
		\end{cases}
	\end{equation}
\end{lem}

\begin{proof}
	See \cite[Lemma 5.1]{ding23}. Although the statement there is formulated for $ n^{-1 + \tau} \le \eta \le \tau^{-1} $, it only requires that $ \eta $ is bounded above. In particular, \eqref{eq:estimate of im m} can be shown from the square-root behavior of $ \varrho $ and the boundedness of $ \operatorname{supp} \varrho $.
\end{proof}

For $ z = E + \i \eta \in \bbC_{+} $, define the resolvent
\begin{equation} \label{eq:R_0}
	R_{0}(z) := (Q_{0} - z)^{-1},
\end{equation}
where $ Q_{0} $ is the sample covariance matrix defined in \eqref{eq:sample covmat} under the non-spiked covariance matrix model. Informally, the local laws state that the resolvent $ R_{0}(z) $ of $ Q_{0} $ is \textit{close} to its (deterministic) counterpart $ \Pi_{R_{0}}(z) $, which is defined by
\begin{equation} \label{eq:counterpart of R_0}
	\Pi_{R_{0}}(z) \equiv \Pi_{R_{0}}(z, \Sigma_{0}) := \frac{-1}{z \left (1 + \phi^{-1/2} m(z) \Sigma_{0} \right )}.
\end{equation}
For a fixed small constant $ \tau \in (0, 1) $, define the spectral domain by
\begin{equation} \label{eq:spectral domain}
	\bfD \equiv \bfD(\tau) := \left \{ z = E + \i \eta \in \bbC: |E - \phi^{1/2} \mathfrak{m}_{1} | \le \tau^{-1}, \ n^{-1 + \tau} \le \eta \le \tau^{-1} \right \}.
\end{equation}
\begin{lem}[Local laws] \label{lem:local law}
	Suppose that Assumption \ref{assump:main} (i)-(iii) holds. Then, for all $ z = E + \i \eta \in \bfD $ uniformly and any deterministic unit vectors $ \bm{\zeta}_{1}, \bm{\zeta}_{2} \in \bbR^{p} $, we have that
	\begin{equation}
		\left | \bm{\zeta}_{1}^{\top} R_{0}(z) \bm{\zeta}_{2} - \bm{\zeta}_{1}^{\top} \Pi_{R_{0}}(z) \bm{\zeta}_{2} \right | \prec \phi^{-1} \Psi(z), \label{eq:local law}
	\end{equation}
	where the error control parameter $ \Psi(z) $ is defined as
	\begin{equation} \label{eq:Psi}
		\Psi(z) := \sqrt{\frac{\im m(z)}{n \eta}} + \frac{1}{n \eta}.
	\end{equation}
\end{lem}

\begin{proof}
	\eqref{eq:local law} has been essentially proved in \cite{ding23, ding24}.
\end{proof}

\begin{rem} \label{rem:simultaneous bound}
	By using the Lipschitz continuity and considering $ n^{-(\alpha+3)} $-net in $ \mathbf{D} $, where $ \alpha $ is defined in \eqref{eq:uhd}, the uniform bound \eqref{eq:local law} in Lemma \ref{lem:local law} can be strengthened to the simultaneous bound for $ z \in \bfD $. This refinement applies to all the local laws stated below. See Remark 2.7 in \cite{BGK17} for details.
\end{rem}

We now introduce two important consequences of Lemma \ref{lem:local law} on the eigenvalues and eigenvectors of $ Q_{0} $. Let $ \lambda_{1} \le \dots \le \lambda_{n} $ be non-trivial eigenvalues of $ Q_{0} $ and $ \{ \bfu_{i} \} $ be the normalized eigenvectors associated with $ \lambda $. The first result is an \emph{eigenvalue rigidity}, which asserts that individual eigenvalues are concentrated around the quantiles $ \{\gamma_{i}\} $ in \eqref{eq:gamma_i}.

\begin{lem}[Rigidity] \label{lem:rigidity}
	Suppose Assumption \ref{assump:main} (i)-(iii) holds. Then,
	\begin{equation} \label{eq:rigidity}
		| \lambda_{i} - \gamma_{i} | \prec \left ( \min \{ i, n+1-i \} \right )^{-1/3} n^{-2/3}
	\end{equation}
	uniformly for $ i \in \llb n \rrb $.
\end{lem}

\begin{proof}
	See \cite[Lemma A.2]{ding24}.
\end{proof}

The second result is an \emph{eigenvector delocalization}. It describes that eigenvectors are typically spread out, with no significant mass concentrated in any direction.

\begin{lem}[Delocalization] \label{lem:delocalization}
	Suppose Assumption \ref{assump:main} (i)-(iii) holds. Then, for any deterministic unit vector $ \bfv \in \bbR^{p} $,
	\begin{equation} \label{eq:delocalization}
		\langle \bfu_{i}, \bfv \rangle^{2} \prec p^{-1}
	\end{equation}
	uniformly for $ i \in \llb n \rrb $.
\end{lem}

\begin{proof}
	The proof is similar to that of \cite[Theorem 3.13]{bloemendal14}. From \eqref{eq:properties of gamma pm}, we can find a sufficiently small $ \tau > 0 $ such that $ | \gamma_{\pm} - \phi^{1/2} \mathfrak{m}_{1} | \le (2\tau)^{-1} $. Set $ z = \lambda_{i} + \i \eta $, where $ \eta = n^{-1 + \tau} $. From Lemma \ref{lem:rigidity} and that $ \gamma_{i} \in [\gamma_{-}, \gamma_{+}] $, there exists a high probability event $ \Xi $ on which $ z \in \bfD $, where $ \bfD $ is defined in \eqref{eq:spectral domain}. By Lemma \ref{lem:local law}, we have
	\[ \bbone(\Xi) \im \left ( \bfv^{\top} R_{0}(z) \bfv \right ) \prec \im \left ( \bfv^{\top} \Pi_{R_{0}}(z) \bfv \right ) + \phi^{-1} \Psi(z). \]
	Here, note that we can still apply Lemma \ref{lem:local law} even though $ z $ is a random variable due to Remark \ref{rem:simultaneous bound}.
	From \eqref{eq:counterpart of R_0}, Lemma \ref{lem:basic estimates of m}, and that  \eqref{eq:properties of gamma pm}, we get
	\[ \im \left ( \bfv^{\top} \Pi_{R_{0}}(z) \bfv \right ) + \phi^{-1} \Psi(z) \le C\phi^{-1} \quad \text{on } \Xi, \]
	where $ C > 0 $ is a generic constant. Therefore,
	\[ \bbone(\Xi) \im \left ( \bfv^{\top} R_{0}(z) \bfv \right ) \prec \phi^{-1}. \]
	We now connect $ \im \left ( \bfv^{\top} R_{0}(z) \bfv \right ) $ with $ | \langle \bfu_{i}, \bfv \rangle |^{2} $ by the following identity
\begin{equation} \label{eq:delocalization proof}
	\im \left ( \bfv^{\top} R_{0}(z) \bfv \right ) = \sum_{j=1}^{p} \frac{\eta | \langle \bfu_{j}, \bfv \rangle |^{2}}{(\lambda_{j} - \lambda_{i})^{2} + \eta^{2}} \ge \frac{\langle \bfu_{i}, \bfv \rangle^{2}}{\eta}.
\end{equation}
	Since $ \eta = n^{-1 + \tau} $, we get $ \bbone(\Xi) \langle \bfu_{i}, \bfv \rangle^{2} \prec n^{-1 + \tau} \phi^{-1} = p^{-1} n^{\tau} $.
\end{proof}

If the real part of $ z $ lies outside the support of $ \varrho $, the local laws can be extended all the way down to the real axis with a good control of the error. For a fixed small constant $ \tau \in (0, 1) $, define
\begin{align} 
	\bfD_{\mathrm{o}} \equiv \bfD_{\mathrm{o}}(\tau) &:= \left \{ z = E + \i \eta \in \bbC: \gamma_{+} + n^{-2/3+\tau} \le E \le \gamma_{+} + \tau^{-1}, \ |\eta| \le \tau^{-1} \right \}, \label{eq:spectral domain outside} \\
	\bfD_{\text{e}} \equiv \bfD_{\text{e}}(\tau) &:= \left \{ z \in \bbC: E \ge \gamma_{+} + n^{-2/3+\tau}, \ |\eta| \le \tau^{-1} \right \}.  \label{eq:spectral domain extended}
\end{align}

\begin{lem}[Local laws outside the support] \label{lem:local law outside}
	Suppose that Assumption \ref{assump:main} (i)-(iii) holds. Let $ \bm{\zeta}_{1}, \bm{\zeta}_{2} \in \bbR^{p} $ be any deterministic unit vectors. Then, we have that
	\begin{equation}
		\left | \bm{\zeta}_{1}^{\top} R_{0}(z) \bm{\zeta}_{2} - \bm{\zeta}_{1}^{\top} \Pi_{R_{0}}(z) \bm{\zeta}_{2} \right | \prec \phi^{-1} (\kappa + \eta)^{-1/4} n^{-1/2} \label{eq:local law outside}
	\end{equation}
	uniformly for $ z = E + \i \eta \in \bfD_{\mathrm{o}} $. Moreover,
	\begin{equation}
		\left | \bm{\zeta}_{1}^{\top} R_{0}(z) \bm{\zeta}_{2} - \bm{\zeta}_{1}^{\top} \Pi_{R_{0}}(z) \bm{\zeta}_{2} \right | \prec \phi^{-1/2} |z|^{-1} \big ( (\kappa+\eta) + (\kappa + \eta)^{1/4} \big )^{-1} n^{-1/2}. \label{eq:local law extended}
	\end{equation}
	uniformly for $ z = E + \i \eta \in \bfD_{\mathrm{e}} $.
\end{lem}

\begin{proof}
	In this proof, we assume $ \eta > 0 $. The case $ \eta = 0 $ follows by taking the limit $ \eta \downarrow 0 $, with the convention $ m(E) := \lim_{\eta \downarrow 0} m(E+\i \eta) $. The case $ \eta < 0 $ follows by symmetry. Since the proof of \eqref{eq:local law outside} follows the standard arguments established in \cite{bloemendal14}, we only highlight the main steps. Define $ \eta_{0} := n^{-1/2} \kappa^{1/4} $. Note that $ \eta_{0} \le \kappa $ on $ \bfD_{\mathrm{o}} $. We focus on the case where $ 0 < \eta < \eta_{0} $, since \eqref{eq:local law outside} easily follows by Lemmas \ref{lem:basic estimates of m} and \ref{lem:local law} if $ \eta \ge \eta_{0} $. Following the lines in the proof of Theorem 3.12 in \cite{bloemendal14}, it suffices to show that 
	\begin{equation} \label{eq:outside proof 1}
			\left | \bm{\zeta}^{\top} R_{0}(z) \bm{\zeta} - \bm{\zeta}^{\top} R_{0}(z_{0}) \bm{\zeta} \right | \prec \phi^{-1} \kappa^{-1/4} n^{-1/2}
	\end{equation}
	and
	\begin{equation} \label{eq:outside proof 2}
		\left | \bm{\zeta}^{\top} \Pi_{R_{0}}(z) \bm{\zeta} - \bm{\zeta}^{\top} \Pi_{R_{0}}(z_{0}) \bm{\zeta} \right | \prec \phi^{-1} \kappa^{-1/4} n^{-1/2},
	\end{equation}
	where $ z_{0} := E+ \i \eta_{0} $.
	
	By (6.4) and (6.5) in \cite{bloemendal14}, where the rigidity (Lemma \ref{lem:rigidity}) is used, \eqref{eq:outside proof 1} holds if one can show
	\begin{equation} \label{eq:outside proof 1.1}
		\im \left ( \bm{\zeta}^{\top} R_{0}(z_{0}) \bm{\zeta} \right ) \prec \phi^{-1} \kappa^{-1/4} n^{-1/2}.
	\end{equation}
	By \eqref{eq:spectrum of Sigma0}, definition of $ \Pi_{R_{0}}(z) $ in \eqref{eq:counterpart of R_0}, and Lemma \ref{lem:basic estimates of m}, we have
	\[ \im \left ( \bm{\zeta}^{\top} \Pi_{R_{0}}(z_{0}) \bm{\zeta} \right ) \asymp | z_{0} |^{-2} \eta_{0} \left ( 1 + (\kappa + \eta_{0})^{-1/2} \right ) \]
	From the definition of $ \eta_{0} $, $ |z_{0}| \asymp \phi^{1/2} $, and that $ \eta_{0} \le \kappa \le \tau^{-1} $, we get
	\[ \im \left ( \bm{\zeta}^{\top} \Pi_{R_{0}}(z_{0}) \bm{\zeta} \right ) \lesssim \phi^{-1} \kappa^{-1/4} n^{-1/2}. \]
	Furthermore, we have $ \Psi(z_{0}) \asymp \kappa^{-1/4} n^{-1/2} $ by Lemma \ref{lem:basic estimates of m}. Then, \eqref{eq:outside proof 1.1} follows from Lemma \ref{lem:local law} with $ z_{0} $.
	
	We now turn to \eqref{eq:outside proof 2}. From \eqref{eq:spectrum of Sigma0}, definition of $ \Pi_{R_{0}}(z) $ in \eqref{eq:counterpart of R_0}, Lemma \ref{lem:basic estimates of m}, that $ |z-z_{0}| \le \eta_{0} $, and that $ |z|, |z_{0}| \asymp \phi^{1/2} $, it suffices to show that
	\begin{equation*}
		|zm(z) - z_{0} m(z_{0})| \prec \phi^{1/2} \kappa^{-1/4} n^{-1/2}.
	\end{equation*}
	By \eqref{eq:spectrum of Sigma0}, \eqref{eq:self-consistent equation}, \eqref{eq:f}, and Lemma \ref{lem:basic estimates of m}, we have
	\[ |(zm(z))'| = \left | \frac{1}{p} \sum_{i=1}^{p} \frac{\phi^{1/2} \sigma_{i} m'(z)}{\left (1+ \phi^{-1/2} \sigma_{i} (m(z) \right )^{2}} \right | \asymp \phi^{1/2} |m'(z)| \]
	for $ z \in \bfD_{\mathrm{o}} $.
	From $ m'(z) = \int \frac{1}{(x-z)^2} \varrho(\dif x) $ and the square-root behavior of $ \varrho $ near the edge (Lemma \ref{lem:properties of rho and f}, we have $ |m'(z)| \lesssim \kappa^{-1/2} $ for $ z \in \bfD_{\text{o}} $. Thus,
	\[ |zm(z) - z_{0}m(z_{0})| \lesssim \phi^{1/2} \kappa^{-1/2} |z - z_0| \le \phi^{1/2} \kappa^{-1/4} n^{-1/2}, \]
	where we use $ |z-z_0| \le \eta_{0} $ for the last step.
	This completes the proof of \eqref{eq:local law outside}.
	
	The proof for \eqref{eq:local law extended} follows the standard application of Helffer-Sjöstrand formula to $ f_{z}(x) = \frac{1}{x-z} + \frac{1}{z} $. See \cite[Proposition 3.8]{bloemendal16}. We omit the details.
\end{proof}

In what follows, we use the slight modifications. Define 
\begin{equation} \label{eq:G and counterpart of G}
	\widetilde{G}(z) := \phi^{1/2} ( 1 + z R_{0}(z) ), \quad \Pi_{\widetilde{G}}(z) := \phi^{1/2} ( 1 + z \Pi_{R_{0}}(z) ) = \frac{m(z) \Sigma_{0}}{1+\phi^{-1/2}m(z) \Sigma_{0}}.
\end{equation}
The following lemma is an immediate consequence of Lemmas \ref{lem:local law} and \ref{lem:local law outside} since $ |z| \asymp \phi^{1/2} $ on $ \bfD $ and $ \bfD_{\mathrm{o}} $.

\begin{lem} \label{lem:G local law}
	Suppose that Assumption \ref{assump:main} (i)-(iii) holds. Let $ \bm{\zeta}_{1}, \bm{\zeta}_{2} \in \bbR^{p} $ and $ \bm{\xi}_{1}, \bm{\xi}_{2} \in \bbR^{n} $ be any deterministic unit vectors.
	\begin{enumerate}
		\item[(i)] For $ z = E + \i \eta \in \bfD $ uniformly,
		\begin{equation}
			\left | \bm{\zeta}_{1}^{\top} \widetilde{G}(z) \bm{\zeta}_{2} - \bm{\zeta}_{1}^{\top} \Pi_{\widetilde{G}}(z) \bm{\zeta}_{2} \right | \prec \Psi(z). \label{eq:G local law}
		\end{equation}
		\item[(ii)] For $ z = E + \i \eta \in \bfD_{\mathrm{o}} $ uniformly,
		\begin{equation}
			\left | \bm{\zeta}_{1}^{\top} \widetilde{G}(z) \bm{\zeta}_{2} - \bm{\zeta}_{1}^{\top} \Pi_{\widetilde{G}}(z) \bm{\zeta}_{2} \right | \prec (\kappa + \eta)^{-1/4} n^{-1/2}. \label{eq:G local law outside}
		\end{equation}
		\item[(iii)] For $ z = E + \i \eta \in \bfD_{\mathrm{e}} $ uniformly,
		\begin{equation}
			\left | \bm{\zeta}_{1}^{\top} \widetilde{G}(z) \bm{\zeta}_{2} - \bm{\zeta}_{1}^{\top} \Pi_{\widetilde{G}}(z) \bm{\zeta}_{2} \right | \prec \big ( (\kappa+\eta) + (\kappa + \eta)^{1/4} \big )^{-1} n^{-1/2}.\label{eq:G local law extended}
		\end{equation}
	\end{enumerate}
\end{lem}


\section{Proof of the main results} \label{app:proof}
This appendix contains the proofs of Theorems \ref{thm:eigenvalue}, \ref{thm:eigenvector}, and \ref{thm:spiked outlier eigenvector projection sum}, following the strategies from \cite{ky13,bloemendal16,bun18,ding21,dy21,dj23,dly24}. Before we start the proof, we introduce several lemmas and notations that will be used throughout the proofs.

We denote
\begin{gather}
	\sfD = \operatorname{diag}(\sfd_{1}, \dots, \sfd_{r},  \allowbreak 0, \dots, 0) \in \bbR^{p \times p}, \quad \sfD_{\mathrm{s}} = \diag{\sfd_{1}, \dots, \sfd_{r}} \in \bbR^{r \times r}, \\
	\Lambda_{\textrm{s}} = \diag{\sigma_{1}, \dots, \sigma_{r}} \in \bbR^{r \times r}, \quad V_{\textrm{s}} = [ \bfv_{1} \ \cdots \ \bfv_{r} ] \in \bbR^{p \times r}.
\end{gather}
With these matrices, we have
\[ \widetilde{\Sigma} = \Sigma_{0} \left ( 1 + V \sfD V^{\top} \right ) = \left ( 1 + V \sfD V^{\top} \right ) \Sigma_{0}, \quad V_{\textrm{s}}^{\top} \Sigma_{0} V_{\textrm{s}} = \Lambda_{\textrm{s}}, \quad V\sfD V^{\top} = V_{\textrm{s}} \sfD_{\textrm{s}} V_{\textrm{s}}^{\top}. \]
Similar to \eqref{eq:R_0}, we define the resolvent
\begin{equation} \label{eq:R tilde}
	\widetilde{R}(z) := ( \widetilde{Q} - z )^{-1}, \quad z \in \bbC_{+},
\end{equation}
where $ \widetilde{Q} $ is the sample covariance matrix defined in \eqref{eq:sample covmat} under the spiked covariance matrix model. Next, we collect some lemmas as in \cite[Section 3.4]{bloemendal16} without proof. Recall the definition of $ \widetilde{G} $ defined in \eqref{eq:G and counterpart of G}.
\begin{lem} \label{lem:perturb 1}
	Suppose that $ \widetilde{\lambda} \in \bbR $ is not an eigenvalue of $ Q_{0} $. Then, $ \widetilde{\lambda} $ is an eigenvalue of $ \widetilde{Q} $ if and only if
	\begin{equation}
		\det \left ( \phi^{1/2} \sfD_{\mathrm{s}}^{-1} + V_{\mathrm{s}}^{\top} \widetilde{G}(\widetilde{\lambda}) V_{\mathrm{s}} \right ) = 0.
	\end{equation}
\end{lem}

\begin{lem} \label{lem:perturb 2}
	For $ z \in \bbC^{+} $,
	\begin{equation}
		V_{\mathrm{s}}^{\top} \widetilde{R}(z) V_{\mathrm{s}} = \frac{1}{z} \left [ \sfD_{\mathrm{s}}^{-1} - \frac{(1 + \sfD_{\mathrm{s}})^{1/2}}{\sfD_{\mathrm{s}}} \left ( \sfD_{\mathrm{s}}^{-1} + \phi^{-1/2} V_{\mathrm{s}}^{\top} \widetilde{G}(z) V_{\mathrm{s}} \right )^{-1} \frac{(1 + \sfD_{\mathrm{s}})^{1/2}}{\sfD_{\mathrm{s}}} \right ].
	\end{equation}
\end{lem}

\begin{lem} \label{lem:interlacing}
	For all $ i \in \llb n \rrb $,
	\begin{equation}
		\widetilde{\lambda}_{i} \in [ \lambda_{i}, \lambda_{i-r} ],
	\end{equation}
	where $ \lambda_{i} = \infty $ for $ i < 1 $ by convention.
\end{lem}
%

\subsection{Proof of Theorem \ref{thm:eigenvalue}} \label{app:spiked eigenvalue proof}
The proof of Theorem \ref{thm:eigenvalue} can be outlined in two steps. First, we construct the permissible region for the eigenvalues of $ \widetilde{Q} $, in the sense that, with high probability, no eigenvalue of $ \widetilde{Q} $ exists outside this region. Second, we show that each connected component of the permissible region contains exactly one eigenvalue of $ \widetilde{Q} $ with high probability.

\noindent\textbf{Proof of Theorem \ref{thm:eigenvalue} (i)}. Recall the notation $ \mathfrak{a}_{i} $ defined in \eqref{eq:a_i}. For a fixed small $ \epsilon > 0 $, define
\begin{equation} \label{eq:permissible region 1}
	\cI_{0} := \left [ 0, \gamma_{+} + n^{-2/3+2\epsilon} \right ] \quad \text{and} \quad \cI_{i} := \left [ \mathfrak{a}_{i} - n^{-1/2 + \epsilon},
	\ \mathfrak{a}_{i} + n^{-1/2 + \epsilon} \right ], \quad i \in \llb r \rrb.
\end{equation}
We claim that $ \{ \cI_{i} \}_{i=0}^{r} $ are disjoint for sufficiently large $ n $. To show this, note that for any $ t \in (\mathsf{c}_{1}, 0) $ and sufficiently large $ n $, we have
\[ f''(t) = -\frac{2}{t^{3}} + \frac{2}{p} \sum_{i=1}^{p} \frac{\phi^{-1/2} \sigma_{i}^{3}}{\left (1 + \phi^{-1/2} \sigma_{i} t \right )^{3}} \asymp |t|^{-3} \gtrsim 1, \]
where we use \eqref{eq:spectrum of Sigma0} in the second step and Lemma \ref{lem:properties of rho and f} (1) in the third step. Then, the mean value theorem with \eqref{eq:spectrum of Sigma0} and \eqref{eq:spike1} gives that 
\begin{equation} \label{eq:a_i separate from support}
	\begin{aligned}
		\mathfrak{a}_{r} - \gamma_{+} &= f(-\phi^{1/2} \widetilde{\sigma}_{r}^{-1}) - f(\sfc_{1}) \\
		&= f''(\xi) (-\phi^{1/2} \widetilde{\sigma}_{r}^{-1}-\mathsf{c}_{1})^{2} \quad \text{for some } \xi \in (\mathsf{c}_{1}, -\phi^{1/2} \widetilde{\sigma}_{r}^{-1}) \\
		&\gtrsim (\varpi \varsigma)^{2},
	\end{aligned}
\end{equation}
where \eqref{eq:spike1} is used. From similar arguments with the fact that $ f' $ is increasing over $ [\mathsf{c}_{1}, 0) $ for sufficiently large $ n $ and \eqref{eq:spike2}, we also have that
\begin{align} \label{eq:a_i disjoint}
	\mathfrak{a}_{i} - \mathfrak{a}_{i+1} = f(-\phi^{1/2} \widetilde{\sigma}_{i}^{-1}) - f(-\phi^{1/2} \widetilde{\sigma}_{i+1}^{-1}) \gtrsim \varpi \varsigma \quad \text{for } i = 1, \dots, r-1,
\end{align}
which completes the proof of the claim.

Next, we show that $ \cI := \sqcup_{i=0}^{r} \cI_{i} $ is the permissible region for the eigenvalues of $ \widetilde{Q} $. By Lemmas \ref{lem:rigidity} and \ref{lem:G local law} (iii) with Remark \ref{rem:simultaneous bound}, there exists a high probability event $ \Xi_{1} \equiv \Xi_{1}(\epsilon) $ such that
\begin{equation} \label{eq:outlier eigenvalue proof 1}
	\bbone (\Xi_{1}) \lambda_{1} \le \gamma_{+} + n^{-2/3 + 2\epsilon},
\end{equation}
\begin{equation} \label{eq:outlier eigenvalue proof 2}
	\bbone (\Xi_{1}) \left \| V_{\mathrm{s}}^{\top} \left ( \widetilde{G}(\widetilde{\lambda}) - \Pi_{\widetilde{G}}(\widetilde{\lambda}) \right ) V_{\mathrm{s}} \right \| \le \left ( \kappa(\widetilde{\lambda}) + \kappa(\widetilde{\lambda})^{1/4} \right )^{-1} n^{-1/2 + \epsilon/2} \quad \text{for } \widetilde{\lambda} \in \bbR \setminus \cI_{0},
\end{equation}
where $ \widetilde{G} $ and $ \Pi_{\widetilde{G}} $ are defined in \eqref{eq:G and counterpart of G}. From now on, we will fix a realization in $ \Xi_{1} $, and the proof will be entirely deterministic. By \eqref{eq:outlier eigenvalue proof 1} and Lemma \ref{lem:perturb 1}, $ \widetilde{\lambda} \notin \cI_{0} $ is an eigenvalue of $ \widetilde{Q} $ if and only if
\begin{align*}
	&  \det \left ( \phi^{1/2} \sfD_{\textrm{s}}^{-1} + V_{\textrm{s}}^{\top} \widetilde{G}(\widetilde{\lambda}) V_{\textrm{s}} \right ) \\
	= \ & \det \left ( \phi^{1/2} \sfD_{\textrm{s}}^{-1} + V_{\textrm{s}}^{\top} \Pi_{\widetilde{G}}(\widetilde{\lambda}) V_{\textrm{s}} + \mathrm{O} \left ( \left ( \kappa(\widetilde{\lambda}) + \kappa(\widetilde{\lambda})^{1/4} \right )^{-1} n^{-1/2 + \epsilon/2} \right ) \right ) \\
	= \ & \det \left ( \phi^{1/2} \sfD_{\textrm{s}}^{-1} + \frac{m(\widetilde{\lambda}) \Lambda_{\mathrm{s}}}{1 + \phi^{-1/2} m(\widetilde{\lambda}) \Lambda_{\mathrm{s}}} + \mathrm{O} \left ( \left ( \kappa(\widetilde{\lambda}) + \kappa(\widetilde{\lambda})^{1/4} \right )^{-1} n^{-1/2 + \epsilon/2} \right ) \right ) \\
	= \ & 0.
\end{align*}
Therefore, it suffices to show that for $ \widetilde{\lambda} \notin \cI $, 
\begin{equation} \label{eq:outlier eigenvalue proof 3}
	\min_{i \in \llb r \rrb} \left | \phi^{1/2} \sfd_{i}^{-1} + \frac{m(\widetilde{\lambda}) \sigma_{i}}{1+\phi^{-1/2} m(\widetilde{\lambda}) \sigma_{i}} \right | \gg \left ( \kappa(\widetilde{\lambda}) + \kappa(\widetilde{\lambda})^{1/4} \right )^{-1} n^{-1/2 + \epsilon/2}.
\end{equation}
Define $ h_{i}(x) = \phi^{1/2} d_{i}^{-1} + \frac{m(x) \sigma_{i}}{1 + \phi^{-1/2} m(x) \sigma_{i}} $ for $ i \in \llb r \rrb $. By using $ m(f(x)) = x $, it is easy to verify that $ h_{i}(\mathfrak{a}_{i}) = 0 $ and that $ h_{i}(x) $ is strictly increasing for $ x > \gamma_{+} $. We now consider two cases to prove \eqref{eq:outlier eigenvalue proof 3}: (i) $ \kappa(\widetilde{\lambda}) \gtrsim 1 $; (ii) $ n^{-2/3+2\epsilon} \le \kappa(\widetilde{\lambda}) \ll 1 $.

\begin{enumerate}[label = (\roman*)]
	\item By monotonicity of $ h_{i} $ above, it suffices to show that $ |h_{i}(\mathfrak{a}_{i} \pm n^{-1/2+\epsilon})| \gg n^{-1/2 + \epsilon/2} $. Note that for $ x \in \cI_{i} $,
	\begin{equation*}
		h_{i}'(x) = \frac{m'(x) \sigma_{i}}{\left (1 + \phi^{-1/2} m(x) \sigma_{i} \right )^{2}} \asymp m'(x) \asymp \frac{1}{f'(m(x))},
	\end{equation*}
	where we use \eqref{eq:spectrum of Sigma0} and \eqref{eq:estimate of m} for the second step and \eqref{eq:self-consistent equation} for the third step. If one can show that $ f'(m(x)) \asymp 1 $ for $ x \in \cI_{i} $, then \eqref{eq:outlier eigenvalue proof 3} follows by the mean value theorem since
	\[ |h_{i}(\mathfrak{a}_{i} \pm n^{-1/2+\epsilon}) | \asymp n^{-1/2 + \epsilon} \gg n^{-1/2 + \epsilon/2}. \]
	It remains to show that $ f'(m(x)) \asymp 1 $ for $ x \in \cI_{i} $. Since $ x \asymp \mathfrak{a}_{i} = f(-\phi^{1/2} \widetilde{\sigma}_{i}^{-1}) $, the mean value theorem with \eqref{eq:self-consistent equation} and \eqref{eq:spike1} gives that
	\[ f'(m(x)) = f'(m(\gamma_{+})) + f''(\xi') (m(x) - m(\gamma_{+})) \asymp 1, \]
	for some $ \xi' \in (m(\gamma_{+}), m(x)) $.
	\item Analogous arguments to (i) shows that $ h_{i}(x) \gtrsim 1 $ for $ \widetilde{\lambda} > \gamma_{+} $ with $ n^{-2/3+2\epsilon} \le \kappa(\widetilde{\lambda}) \ll 1 $. Thus, \eqref{eq:outlier eigenvalue proof 3} follows since the right-hand side of \eqref{eq:outlier eigenvalue proof 3} is of order $ \kappa(\widetilde{\lambda})^{-1/4} n^{-1/2 + \epsilon/2} \le n^{-1/3} $ at most.
\end{enumerate}

Finally, we prove that each $ \mathcal{I}_{i} $, $ i \in \llb r \rrb $, contains exactly one eigenvalue of $ \widetilde{Q} $. Since $ \mathfrak{a}_{i} $ are separated by a constant-order gap for sufficiently large $ n $, as shown in \eqref{eq:a_i disjoint}, we can choose a positively oriented circle $ \cC_{i} $ centered at $ \mathfrak{a}_{i} $ with radius of order one such that it contains no other $ \mathfrak{a}_{j} $, $ j \neq i $. Define
\[ \cM(z) := \det \left ( \phi^{1/2} \sfD_{\textrm{s}}^{-1} + V_{\textrm{s}}^{\top} \widetilde{G}(z) V_{\textrm{s}} \right ), \quad  \cL(z) := \det \left ( \phi^{1/2} \sfD_{\textrm{s}}^{-1} + \frac{m(z) \Lambda_{\mathrm{s}}}{1 + \phi^{-1/2} m(z) \Lambda_{\mathrm{s}}} \right ). \]
By following the arguments above, we have that on the event $ \Xi_{0} $, there exists $ \delta > 0 $ such that
\[ \min_{z \in \cC_{i}} | \cL(z) | \ge \delta > 0 \quad \text{and} \quad |\cM(z) - \cL(z)| \le n^{-1/2 + \epsilon / 2}. \]
Since $ \mathfrak{a}_{i} $ is the only zero of $ \cL $ inside $ \cC_{i} $, Rouché's theorem gives that $ \mathsf{I}_{j} $ contains exactly one eigenvalue of $ \widetilde{Q} $. This completes the proof of \eqref{eq:outlier eigenvalue}.

\noindent \textbf{Proof of Theorem \ref{thm:eigenvalue} (ii)}. To prove \eqref{eq:non-outlier eigenvalue}, we establish the eigenvalue sticking estimate
\begin{equation} \label{eq:eigenvalue sticking}
	\left | \widetilde{\lambda}_{i} - \lambda_{i-r} \right | \prec n^{-1}.
\end{equation}
Together with the rigidity estimate in Lemma \ref{lem:rigidity}, \eqref{eq:non-outlier eigenvalue} follows from \eqref{eq:eigenvalue sticking}. The rest of this subsection is devoted to the proof of \eqref{eq:eigenvalue sticking}.

Let $ \epsilon > 0 $ be a sufficiently small fixed constant such that $ \gamma_{+} + 1 \in \bfD \equiv \bfD(\epsilon) $, where $\bfD$ is defined in \eqref{eq:spectral domain}. Note that this is possible due to Lemma \ref{lem:properties of rho and f} (ii). Then, from Lemma \ref{lem:rigidity}, Lemma \ref{lem:G local law} (i) with Remark \ref{rem:simultaneous bound}, and Theorem \ref{thm:eigenvalue} (i) with Lemma \ref{lem:interlacing}, there exists a high probability event $ \Xi_{2} \equiv \Xi_{2}(\epsilon) $ such that
\begin{gather} 
	\bbone (\Xi_{2}) |\lambda_{i} - \gamma_{i}| \le ( \min \{i, n+1-i\} )^{-1/3} n^{-2/3 + \epsilon/2}, \label{eq:bulk eigenvalue proof 1} \\
	\bbone (\Xi_{2}') \left \| V_{\mathrm{s}}^{\top} \left ( \widetilde{G}(z) - \Pi_{\widetilde{G}}(z) \right ) V_{\mathrm{s}} \right \| \le n^{\epsilon/2} \Psi(z) \text{ for } z \in \bfD, \label{eq:bulk eigenvalue proof 2} \\
	\bbone (\Xi_{2}) |\widetilde{\lambda}_{r+1} - \gamma_{+}| \le n^{-2/3 + \epsilon/2}, \label{eq:bulk eigenvalue proof 3}
\end{gather}
where we recall the definition of $ \Psi(z) $ in \eqref{eq:Psi}. We fix a realization in $ \Xi_{2} $.

For $ \min \{i, n+1-i\} > n^{1 - 2\epsilon} $, \eqref{eq:eigenvalue sticking} follows from Lemma \ref{lem:rigidity}, Lemma \ref{lem:interlacing}, and \eqref{eq:bulk eigenvalue proof 1}, since
\[ | \widetilde{\lambda}_{i} - \lambda_{i-r} | \le | \lambda_{i} - \lambda_{i-r} | \le C ( \min \{i, n+1-i\} )^{-1/3} n^{-2/3 + \epsilon} \le C n^{-1 + 2\epsilon}. \]

Therefore, we now consider the case $ \min \{i, n+1-i\} \le n^{1 - 2\epsilon} $. By symmetry, we may assume that $ i \le n^{1-2\epsilon} $. Define
\begin{equation} \label{eq:J_i}
	\mathcal{J}_{i} := \left \{ \widetilde{\lambda} \in \left [ \lambda_{i-r-1}, \ \gamma_{+} + n^{-2/3 + 2 \epsilon} \right ]: \operatorname{dist}(x, \sigma(Q_{0})) > n^{-1+2\epsilon} \right \},
\end{equation}
where $ \sigma(Q_{0}) $ is the spectrum of $ Q_{0} $. We claim that $ \mathcal{J}_{i} $ contains no eigenvalue of $ \widetilde{Q} $. To show this, let $ \eta_{0} = n^{-1+2\epsilon} $. By Lemma \ref{lem:perturb 1}, $ \widetilde{\lambda} \in \cJ_{i} $ is an eigenvalue of $ \widetilde{Q} $ if and only if 
\begin{equation} \label{eq:bulk eigenvalue proof det}
	\begin{aligned}
		0 &= \det \left (\phi^{1/2} \sfD_{\mathrm{s}}^{-1} + V_{\mathrm{s}}^{\top} \widetilde{G}(\widetilde{\lambda}) V_{\mathrm{s}} \right ) \\
		&= \det \bigg (\phi^{1/2} \sfD_{\mathrm{s}}^{-1} + V_{\mathrm{s}}^{\top} \Pi_{\widetilde{G}}(\widetilde{\lambda}) V_{\mathrm{s}} + V_{\mathrm{s}}^{\top} \left ( \Pi_{\widetilde{G}}(\widetilde{\lambda} + \i \eta_{0}) - \Pi_{\widetilde{G}}(\widetilde{\lambda}) \right ) V_{\mathrm{s}} \\
		& \qquad \qquad \qquad + V_{\mathrm{s}}^{\top} \left ( \widetilde{G}(\widetilde{\lambda} + \i \eta_{0}) - \Pi_{\widetilde{G}}(\widetilde{\lambda} + \i \eta_{0}) \right ) V_{\mathrm{s}} + V_{\mathrm{s}}^{\top} \left ( \widetilde{G}(\widetilde{\lambda}) - \widetilde{G}(\widetilde{\lambda} + \i \eta_{0}) \right ) V_{\mathrm{s}} \bigg ).
	\end{aligned}
\end{equation}
Note that for any $ j, k \in \llb r \rrb $, we have that
\begin{equation} \label{eq:bulk eigenvalue proof inter 1}
	\left | \bfv_{j}^{\top} \left ( \Pi_{\widetilde{G}}(\widetilde{\lambda} + \i \eta_{0}) - \widetilde{G}(\widetilde{\lambda}) \right ) \bfv_{k} \right | \lesssim \im m(\widetilde{\lambda} + \i \eta_{0}),
\end{equation}
which is from Lemma \ref{lem:basic estimates of m} and the definition of $ \cJ_{i} $, and that
\begin{align}
	\left | \bfv_{j}^{\top} \left ( \widetilde{G}(\widetilde{\lambda} + \i \eta_{0}) - \widetilde{G}(\widetilde{\lambda}) \right ) \bfv_{k} \right | &= \phi^{1/2} \left | \widetilde{\lambda} \bfv_{j}^{\top} \left ( R_{0}(\widetilde{\lambda} + \i \eta_{0}) - R_{0}(\widetilde{\lambda}) \right ) \bfv_{k} + \i \eta_{0} \bfv_{j}^{\top} R_{0}(\widetilde{\lambda} + \i \eta_{0}) \bfv_{k}  \right | \notag \\
	&\le C \phi \max_{j \in \llb r \rrb} \im \bfv_{j}^{\top} R_{0}(\widetilde{\lambda} + \i \eta_{0}) \bfv_{j} \notag \\
	&\le C \phi \max_{j \in \llb r \rrb} \im \bfv_{j}^{\top} \Pi_{R_{0}}(\widetilde{\lambda} + \i \eta_{0}) \bfv_{j} + n^{\epsilon/2} \Psi(\widetilde{\lambda} + \i \eta_{0}) \notag \\
	&\le C \eta_{0} + n^{\epsilon/2} \Psi (\widetilde{\lambda} + \i \eta_{0}) \notag \\
	&\le n^{\epsilon/2} \Psi (\widetilde{\lambda} + \i \eta_{0}), \label{eq:bulk eigenvalue proof inter 2}
\end{align}
where we use \eqref{eq:bulk eigenvalue proof 2} for the third step. From \eqref{eq:bulk eigenvalue proof 2}, \eqref{eq:bulk eigenvalue proof inter 1}, \eqref{eq:bulk eigenvalue proof inter 2}, and the estimate $ \Psi(E+\i\eta) \lesssim \im m(E+\i\eta) + (n\eta)^{-1} $, \eqref{eq:bulk eigenvalue proof det} can be rewritten as
\[ 0 = \det \left (\phi^{1/2} \sfD_{\mathrm{s}}^{-1} + V_{\mathrm{s}}^{\top} \widetilde{G}(\widetilde{\lambda}) V_{\mathrm{s}} \right ) = \det \left (\phi^{1/2} \sfD_{\mathrm{s}}^{-1} + V_{\mathrm{s}}^{\top} \Pi_{\widetilde{G}}(\widetilde{\lambda}) V_{\mathrm{s}} + \mathrm{O} \left ( n^{\epsilon/2} \im m(\widetilde{\lambda} + \i \eta_{0}) + \frac{n^{\epsilon/2}}{n\eta_{0}} \right ) \right ). \]
Therefore, $ \widetilde{\lambda} \in \cJ_{i} $ is not an eigenvalue of $ \widetilde{Q} $ if one can show that
\begin{equation} \label{eq:bulk eigenvalue proof final}
	\min_{j \in \llb r \rrb} \left | \phi^{1/2} \sfd_{j}^{-1} + \frac{m(\widetilde{\lambda}) \sigma_{j}}{1+\phi^{-1/2} m(\widetilde{\lambda}) \sigma_{j}} \right | \gg n^{\epsilon/2} \left ( \im m(\widetilde{\lambda} + \i \eta_{0}) + \frac{1}{n\eta_{0}} \right ).
\end{equation}

To estimate the left-hand side of \eqref{eq:bulk eigenvalue proof final}, observe that
\begin{align*}
	\left | \phi^{1/2} \sfd_{j}^{-1} + \frac{m(\widetilde{\lambda}) \sigma_{j}}{1+\phi^{-1/2} m(\widetilde{\lambda}) \sigma_{j}} \right | = \left | \frac{\widetilde{\sigma}_{j} \left ( m(\widetilde{\lambda}) + \phi^{1/2} \widetilde{\sigma}_{j}^{-1} \right )}{\sfd_{j} \left ( 1 + \phi^{-1/2} m(\widetilde{\lambda}) \sigma_{j} \right )} \right | \asymp | m(\widetilde{\lambda}) - m(\fa_{j}) |.
\end{align*}
It is relatively easy to show that for all $ j \in \llb r \rrb $
\begin{equation} \label{eq:bulk eigenvalue proof final 2}
	| m(\widetilde{\lambda}) - m(\fa_{j}) | \ge c_{0} \text{ if } \widetilde{\lambda} \ge \gamma_{+}
\end{equation}
for some fixed constant $ c_{0} $ from the similar argument to \eqref{eq:a_i separate from support}. Therefore, we focus on the case $ \widetilde{\lambda} < \gamma_{+} $. From the definition of $ \cJ_{i} $, 
\begin{equation} \label{eq:bulk eigenvalue proof kappa}
	\gamma_{+} - \widetilde{\lambda} \le \gamma_{+} - \lambda_{i-r-1} \le | \gamma_{+} - \gamma_{i-r-1}| + | \gamma_{i-r-1} - \lambda_{i-r-1}| \le \left ( \frac{i}{n} \right )^{2/3} + n^{-2/3 + \epsilon/2} \le n^{-4\epsilon/3},
\end{equation}
where we use the square-root behavior and \eqref{eq:bulk eigenvalue proof 1} for the second step.
Then, the square-root behavior gives
\begin{equation}
	| m(\gamma_{+}) - m(\widetilde{\lambda}) | \asymp (\gamma_{+} - \widetilde{\lambda})^{1/2} \le n^{-2\epsilon/3}.
\end{equation}
Combining this with \eqref{eq:bulk eigenvalue proof final 2}, we have
\begin{equation} \label{eq:bulk eigenvalue proof final 3}
	| m(\widetilde{\lambda}) - m (\fa_{j}) | \ge | m(\gamma_{+}) - m(\fa_{j}) | - | m(\gamma_{+}) - m(\widetilde{\lambda}) | \ge c_{0}/2.
\end{equation}
From \eqref{eq:bulk eigenvalue proof final 2} and \eqref{eq:bulk eigenvalue proof final 3}, we have that for $ \widetilde{\lambda} \in \cJ_{i} $
\begin{equation} \label{eq:bulk eigenvalue proof final 4}
	\min_{j \in \llb r \rrb} \left | \phi^{1/2} \sfd_{j}^{-1} + \frac{m(\widetilde{\lambda}) \sigma_{j}}{1+\phi^{-1/2} m(\widetilde{\lambda}) \sigma_{j}} \right | \gtrsim 1.
\end{equation}

Next, we estimate the right-hand side of \eqref{eq:bulk eigenvalue proof final}. From Lemma \ref{lem:basic estimates of m}, \eqref{eq:bulk eigenvalue proof kappa} and the choice of $ \eta_{0} $, for $ \widetilde{\lambda} \le \gamma_{+} $, we have
\begin{equation} \label{eq:bulk eigenvalue proof final 5}
	n^{\epsilon/2} \left ( \im m(\widetilde{\lambda} + \i \eta_{0}) + \frac{1}{n\eta_{0}} \right ) \asymp n^{\epsilon/2} \left ( \sqrt{\kappa(\widetilde{\lambda}) + \eta_{0}} + \frac{1}{n\eta_{0}} \right ) \le n^{\epsilon/2} \left ( (\gamma_{+} - \widetilde{\lambda})^{1/2} + \eta_{0}^{1/2} + \frac{1}{n\eta_{0}} \right ) \ll 1.
\end{equation}
Similarly, from Lemma \ref{lem:basic estimates of m}, the estimate $ \eta_{0}/\sqrt{\kappa+\eta_{0}} \lesssim \sqrt{\eta_{0}} $, and the choice of $ \eta_{0} $, for $ \widetilde{\lambda} > \gamma_{+} $, we have
\begin{equation} \label{eq:bulk eigenvalue proof final 6}
	n^{\epsilon/2} \left ( \im m(\widetilde{\lambda} + \i \eta_{0}) + \frac{1}{n\eta_{0}} \right ) \asymp n^{\epsilon/2} \left ( \frac{\eta_{0}}{\sqrt{\kappa(\widetilde{\lambda}) + \eta_{0}}} + \frac{1}{n\eta_{0}} \right ) \lesssim n^{\epsilon/2} \left ( \eta_{0}^{1/2} + \frac{1}{n\eta_{0}} \right ) \ll 1.
\end{equation}
\eqref{eq:bulk eigenvalue proof final} follows by \eqref{eq:bulk eigenvalue proof final 4}-\eqref{eq:bulk eigenvalue proof final 6}.

The next step is the counting argument as in the previous proof. Since it is essentially same as the one discussed in \cite[Proposition 6.8]{ky13}, we omit the details.

\subsection{Proof of Theorem \ref{thm:eigenvector}} \label{app:spiked eigenvector proof}
In the proof, we fix $ i $. \\
\noindent\textbf{Proof of Theorem \ref{thm:eigenvector} (i)}. By \eqref{eq:spike1}, there exists some fixed small constant $ \tau \equiv \tau(\varpi) > 0 $ such that 
\begin{equation} \label{eq:choice of tau}
	f \left ( -\phi^{1/2} \widetilde{\sigma}_{i}^{-1} + \frac{\varpi}{3} \right ) < \gamma_{+} + \tau^{-1}
\end{equation}
for sufficiently large $ n $. We start with $ j \in \llb r \rrb $. From Lemma \ref{lem:rigidity}, Lemma \ref{lem:G local law} (ii) with Remark \ref{rem:simultaneous bound}, and Theorem \ref{thm:eigenvalue},  for a fixed small constant $ \epsilon > 0 $, there exists a high probability event $ \Xi \equiv \Xi(\epsilon, \varpi) $ such that
\begin{gather}
	\bbone(\Xi) \lambda_{1} \le \gamma_{+} + n^{-2/3 + \epsilon}, \label{eq:outlier eigenvector 1 proof 1} \\
	\bbone(\Xi) \left \| V_{\mathrm{s}}^{\top} \left ( \widetilde{G}(\widetilde{\lambda}) - \Pi_{\widetilde{G}}(\widetilde{\lambda}) \right ) V_{\mathrm{s}} \right \| \le n^{-1/2 + \epsilon} \quad \text{for all } z \in \bfD_{\mathrm{o}} \equiv \bfD_{\mathrm{o}}(\tau), \label{eq:outlier eigenvector 1 proof 2} \\
	\bbone(\Xi) \left | \widetilde{\lambda}_{i} - \mathfrak{a}_{i} \right | \le n^{-1/2 + \epsilon} \quad \text{for all } i \in \llb r \rrb, \label{eq:outlier eigenvector 1 proof 3}
\end{gather}
where we recall the definition of $ \bfD_{\mathrm{o}}(\tau) $ in \eqref{eq:spectral domain outside}. From now on, we fix a realization in $ \Xi $. 

Let $ \Gamma_{j} $ be the circular contour of radius $ \varpi/3 $ centered at $ -\phi^{1/2} \widetilde{\sigma}_{j}^{-1} $. By \eqref{eq:outlier eigenvector 1 proof 3} and the similar argument to \eqref{eq:a_i disjoint}, we have that 
\begin{equation}
	\widetilde{\lambda}_{j} \in f(\Gamma_{i}) \text{ if and only if } j = i, \label{eq:outlier eigenvector 1 proof contour 1}
\end{equation}
From the choice of $ \tau $ in \eqref{eq:choice of tau}, we also have
\begin{equation}
	f(\Gamma_{i}) \subset \bfD_{\mathrm{o}}. \label{eq:outlier eigenvector 1 proof contour 2}
\end{equation}
We now provide the integral representation of $ \langle \widetilde{\bfu}_{i}, \bfv_{j} \rangle^{2} $. By Cauchy's integral theorem and the residue theorem with \eqref{eq:outlier eigenvector 1 proof contour 1} and the spectral decomposition of $ \widetilde{R}(z) $ defined in \eqref{eq:R tilde},
\[ \langle \widetilde{\bfu}_{i}, \bfv_{j} \rangle^{2} = - \frac{1}{2\pi \i} \oint_{f(\Gamma_{i})} \left  \langle \bfv_{j}, \widetilde{R}(z) \bfv_{j} \right \rangle \dif z. \]
Moreover, from Lemma \ref{lem:perturb 2} and Cauchy's integral theorem,
\[ \langle \widetilde{\bfu}_{i}, \bfv_{j} \rangle^{2} = \frac{1}{2\pi \i} \frac{1+\sfd_{j}}{\sfd_{j}^{2}} \oint_{f(\Gamma_{i})} \left ( \sfD_{\mathrm{s}}^{-1} + \phi^{-1/2} V_{\mathrm{s}}^{\top} \widetilde{G}(z) V_{\mathrm{s}} \right )_{jj}^{-1} \frac{\dif z}{z}. \]
Define a deterministic diagonal matrix $ \Upsilon(z) \in \bbC^{r \times r} $ and a random error matrix $ \Delta(z) \in \bbC^{r \times r} $ by
\begin{align}
	\Upsilon &\equiv \Upsilon(z) := \sfD_{\mathrm{s}}^{-1} + \phi^{-1/2} V_{\mathrm{s}}^{\top} \Pi_{\widetilde{G}}(z) V_{\mathrm{s}} = \sfD_{\mathrm{s}}^{-1} + \frac{\phi^{-1/2} m(z) \Lambda_{\mathrm{s}}}{1+\phi^{-1/2} m(z) \Lambda_{\mathrm{s}}}, \label{eq:outlier eigenvector 1 leading} \\
	\Delta &\equiv \Delta(z) := \phi^{-1/2} V_{\mathrm{s}}^{\top} \left ( \Pi_{\widetilde{G}}(z) - \widetilde{G}(z) \right ) V_{\mathrm{s}} = \frac{\phi^{-1/2} m(z) \Lambda_{\mathrm{s}}}{1+\phi^{-1/2} m(z) \Lambda_{\mathrm{s}}} - \phi^{-1/2} V_{\mathrm{s}}^{\top} \widetilde{G}(z) V_{\mathrm{s}}. \label{eq:outlier eigenvector 1 error}
\end{align}
From the resolvent expansion
\begin{equation} \label{eq:resolvent expansion}
	( \Upsilon - \Delta )^{-1} = \Upsilon^{-1} + ( \Upsilon - \Delta )^{-1} \Delta \Upsilon^{-1},
\end{equation}
we can  write
\[ \langle \widetilde{\bfu}_{i}, \bfv_{j} \rangle^{2} = s_{1} + s_{2}, \]
where
\begin{align*}
	s_{1} &= \frac{1}{2\pi \i} \frac{1+\sfd_{j}}{\sfd_{j}^{2}} \oint_{f(\Gamma_{i})} \left [ \Upsilon(z)^{-1} \right ]_{jj} \frac{\dif z}{z}, \\
	s_{2} &= \frac{1}{2\pi \i} \frac{1+\sfd_{j}}{\sfd_{j}^{2}} \oint_{f(\Gamma_{i})} \left [ \left ( \sfD_{\mathrm{s}}^{-1} + \phi^{-1/2} V_{\mathrm{s}}^{\top} \widetilde{G}(z) V_{\mathrm{s}} \right )^{-1} \Delta(z) \Upsilon^{-1}(z) \right ]_{jj} \frac{\dif z}{z}.
\end{align*}
	
For $ s_{1} $, notice that
\begin{equation} \label{eq:diagonal of Upsilon inverse}
	\left  [ \Upsilon(z)^{-1} \right  ]_{jj} = \left ( \sfd_{j}^{-1} + \frac{\phi^{-1/2} m(z) \sigma_{j}}{1 + \phi^{-1/2} m(z) \sigma_{j}} \right )^{-1} = \frac{\sfd_{j}}{1+\sfd_{j}} \cdot \frac{m(z) + \phi^{1/2} \sigma_{j}^{-1}}{m(z) + \phi^{1/2} \widetilde{\sigma}_{j}^{-1}}.
\end{equation}
With $ z = f(\zeta) $ and \eqref{eq:self-consistent equation}, $ s_{1} $ can be written as
\[ s_{1} = \frac{1}{2\pi \i} \frac{1}{\sfd_{j}} \oint_{\Gamma_{i}} \frac{\zeta + \phi^{1/2} \sigma_{j}^{-1}}{\zeta + \phi^{1/2} \widetilde{\sigma}_{j}^{-1}} \frac{f'(\zeta)}{f(\zeta)} \dif \zeta. \]
The integrand is holomorphic if $ j \neq i $ and has a simple pole at $ -\phi^{1/2} \widetilde{\sigma}_{i}^{-1} $ if $ j = i $ by the construction of $ \Gamma_{i} $ above. Therefore, Cauchy's integral theorem and the residue theorem give that
\begin{equation} \label{eq:s1}
	s_{1} = \delta_{ij} \phi^{1/2} \widetilde{\sigma}_{i}^{-1} \frac{f'(-\phi^{1/2} \widetilde{\sigma}_{i}^{-1})}{f(-\phi^{1/2} \widetilde{\sigma}_{i}^{-1})} = \delta_{ij} \mathfrak{b}_{i}.
\end{equation}

Similarly, we can write
\[ s_{2} = \frac{1}{2\pi \i} \frac{1+\sfd_{j}}{\sfd_{j}^{2}} \oint_{\Gamma_{i}} \left [ \left ( \sfD_{\mathrm{s}}^{-1} +  \phi^{-1/2} V_{\mathrm{s}}^{\top} \widetilde{G}(f(\zeta)) V_{\mathrm{s}} \right )^{-1} \Delta(f(\zeta)) \Upsilon^{-1}(f(\zeta)) \right ]_{jj} \frac{f'(\zeta)}{f(\zeta)} \dif \zeta. \]
From \eqref{eq:spectrum of Sigma0}, \eqref{eq:spike1}, \eqref{eq:diagonal of Upsilon inverse}, and \eqref{eq:outlier eigenvector 1 proof 2}, we have
\begin{equation} \label{eq:outlier eigenvector 1 error proof 1}
	\left [ \Upsilon(f(\zeta))^{-1} \right ]_{jj} \asymp \phi^{1/2} \text{ and } \left \| \sfD_{\mathrm{s}}^{-1} +  \phi^{-1/2} V_{\mathrm{s}}^{\top} \widetilde{G}(f(\zeta)) V_{\mathrm{s}} \right \|^{-1} \asymp \phi^{1/2} \text{ for } \zeta \in \Gamma_{i}.
\end{equation}
Therefore,
\begin{equation} \label{eq:s2}
	|s_{2}| \le C \phi^{-1/2} n^{-1/2 + \epsilon} = C p^{-1/2} n^{\epsilon},
\end{equation}
where we use \eqref{eq:outlier eigenvector 1 proof 2}, \eqref{eq:outlier eigenvector 1 proof contour 2}, \eqref{eq:outlier eigenvector 1 error}, and the basic estimates $ \sfd_{j}, f(\zeta) \asymp \phi^{1/2} $ and $ f'(\zeta) \asymp 1 $. By \eqref{eq:s1} and \eqref{eq:s2}, we complete the proof of \eqref{eq:outlier eigenvector} for $ j \in \llb r \rrb $.

Now, we consider $ j \in \llb r+1, p \rrb $. We hereby prove the following for the future use
\begin{equation} \label{eq:outlier eigenvector stronger}
	\left | \langle \widetilde{\mathbf{u}}_{i}, \mathbf{v}_{j} \rangle^{2} - \frac{\mathsf{d}_{i}^{2} \mathfrak{b}_{i}}{\phi (1+\sfd_{i})} \langle \bfv_{j}, \widetilde{G}(\mathfrak{a}_{i}) \bfv_{i} \rangle \langle \bfv_{i}, \widetilde{G}(\mathfrak{a}_{i}) \bfv_{j} \rangle \right | = \mathrm{O}_{\prec} \left ( p^{-1} n^{-1/2} \right ).
\end{equation}
If \eqref{eq:outlier eigenvector stronger} holds, then \eqref{eq:outlier eigenvector} for $ j \in \llb r+1, p \rrb $ follows by Lemma \ref{lem:G local law} (iii). The rest is devoted to prove \eqref{eq:outlier eigenvector stronger}. For $ \delta > 0 $, define
\begin{gather*}
	\sfD_{j, \delta} := \diag{ \sfd_{1}, \dots, \sfd_{r}, \delta } \in \bbR^{(r+1) \times (r+1)}, \\
	\Lambda_{j} := \diag{\sigma_{1}, \dots, \sigma_{r}, \sigma_{j}} \in \bbR^{(r+1) \times (r+1)}, \\
	V_{j} := [ \bfv_{1}, \cdots, \bfv_{r}, \bfv_{j} ] \in \bbR^{p \times (r+1)}, \\
	\widetilde{\Sigma}_{j, \delta} := \widetilde{\Sigma} + \delta \sigma_{j} \bfv_{j} \bfv_{j}^{\top} \in \bbR^{p \times p}.
\end{gather*}
Similar to \eqref{eq:outlier eigenvector 1 proof 1}-\eqref{eq:outlier eigenvector 1 proof 3}, we consider a high probability event $ \Xi' \equiv \Xi'(\epsilon, \varpi) $ such that
\begin{gather}
	\bbone(\Xi') \lambda_{1} \le \gamma_{+} + n^{-2/3 + \epsilon}, \label{eq:outlier eigenvector 2 proof 1} \\
	\bbone(\Xi') \left | V_{j}^{\top} \left ( \widetilde{G}(\widetilde{\lambda}) - \Pi_{\widetilde{G}}(\widetilde{\lambda}) \right ) V_{j} \right | \le n^{-1/2 + \epsilon} \quad \text{for all } z \in \bfD_{\mathrm{o}}, \label{eq:outlier eigenvector 2 proof 2} \\
	\bbone(\Xi') \left | \widetilde{\lambda}_{i} - \mathfrak{a}_{i} \right | \le n^{-1/2 + \epsilon} \quad \text{for all } i \in \llb r \rrb. \label{eq:outlier eigenvector 2 proof 3}
\end{gather}
We now fix a realization in $ \Xi' $ and consider
\[ \widetilde{R}_{j, \delta}(z) := \left (\widetilde{\Sigma}_{j, \delta}^{1/2} XX^{\top} \widetilde{\Sigma}_{j, \delta}^{1/2} - z \right )^{-1}. \]
As in Lemma \ref{lem:perturb 2}, $ V_{l}^{\top} \widetilde{R}_{l, \delta}(z) V_{l} $ can be written as
\begin{equation}
	V_{j}^{\top} \widetilde{R}_{j, \delta}(z) V_{j} = \frac{1}{z} \left [ \sfD_{j, \delta}^{-1} - \frac{(1 + \sfD_{j, \delta})^{1/2}}{\sfD_{j, \delta}} \left ( \sfD_{j, \delta}^{-1} + \phi^{-1/2} V_{j}^{\top} \widetilde{G}(z) V_{j} \right )^{-1} \frac{(1 + \sfD_{j, \delta})^{1/2}}{\sfD_{j, \delta}} \right ].
\end{equation}
Recall the contour $ \Gamma_{i} $ define above. Then,
\begin{align*}
	\langle \widetilde{\mathbf{u}}_{i}, \mathbf{v}_{j} \rangle^{2} &= -\frac{1}{2\pi \i} \oint_{f(\Gamma_{i})} \left \langle \bfv_{j}, \widetilde{R}(z) \bfv_{j} \right \rangle \dif z \\
	&= \lim_{\delta \downarrow 0} - \frac{1}{2\pi \i} \oint_{f(\Gamma_{i})} \left \langle \bfv_{j}, \widetilde{R}_{j, \delta}(z) \bfv_{j} \right \rangle \dif z \\
	&= \lim_{\delta \downarrow 0} \frac{1}{2\pi \i} \frac{1 + \delta}{\delta^{2}} \oint_{f(\Gamma_{i})} \left [ \left ( D_{l, \delta}^{-1} + V_{l}^{\top} \widetilde{G}(z) V_{l} \right )^{-1} \right ]_{r+1, r+1} \frac{\dif z}{z},
\end{align*}
where we use the residue theorem and the Cauchy's integral theorem.
Define $ \Upsilon_{j, \delta} $ and $ \Delta_{j} $ similarly to those in \eqref{eq:outlier eigenvector 1 leading} and \eqref{eq:outlier eigenvector 1 error}:
\begin{align}
	\Upsilon_{j, \delta} &\equiv \Upsilon_{j, \delta}(z) := \sfD_{j, \delta}^{-1} + \frac{\phi^{-1/2} m(z) \Lambda_{j}}{1 + \phi^{-1/2} m(z) \Lambda_{j}}, \label{eq:outlier eigenvector 2 leading} \\
	\Delta_{j} &\equiv \Delta_{j}(z) := \frac{\phi^{-1/2} m(z) \Lambda_{j}}{1 + \phi^{-1/2} m(z) \Lambda_{j}} - \phi^{-1/2} V_{j}^{\top} \widetilde{G}(z) V_{j}. \label{eq:outlier eigenvector 2 error}
\end{align}
Note that $ [\Upsilon_{j,\delta}(z)^{-1}]_{kk} $, $ k \in \llb r \rrb $, has the same form as in \eqref{eq:diagonal of Upsilon inverse}, and 
\begin{equation} \label{eq:diagonal of Upsilon inverse 2}
	[\Upsilon_{j,\delta}(z)^{-1}]_{r+1, r+1} = \frac{\delta}{1+\delta} \cdot \frac{m(z) + \phi^{1/2} \sigma_{j}^{-1}}{m(z) + \phi^{1/2} (1+\delta)^{-1} \sigma_{j}^{-1}}.
\end{equation}
From the expansion
\begin{equation}
	( \Upsilon_{j, \delta} - \Delta_{j} )^{-1} = \Upsilon_{j, \delta}^{-1} + \Upsilon_{j, \delta}^{-1} \Delta_{j} \Upsilon_{j, \delta}^{-1} + \Upsilon_{j, \delta}^{-1} \Delta_{j} \Upsilon_{j, \delta}^{-1} \Delta_{j} \Upsilon_{j, \delta}^{-1} + \Upsilon_{j, \delta}^{-1} \Delta_{j} \Upsilon_{j, \delta}^{-1} \Delta_{j} ( \Upsilon_{j, \delta} - \Delta_{j} )^{-1} \Delta_{j} \Upsilon_{j, \delta}^{-1}
\end{equation}
and the change of variable $ z = f(\zeta) $,
\[ \langle \widetilde{\mathbf{u}}_{i}, \mathbf{v}_{j} \rangle^{2} = t_{1} + t_{2} + t_{3} + t_{4}, \]
where
\begin{align*}
	t_{1} &= \lim_{\delta \downarrow 0} \frac{1}{2\pi \i} \frac{1+ \delta}{\delta^{2}} \oint_{\Gamma_{i}} \left [ \Upsilon_{j, \delta}(f(\zeta))^{-1} \right ]_{r+1, r+1} \frac{f'(\zeta)}{f(\zeta)} \dif \zeta, \\
	t_{2} &= \lim_{\delta \downarrow 0} \frac{1}{2\pi \i} \frac{1+ \delta}{\delta^{2}} \oint_{\Gamma_{i}} \left [ \Upsilon_{j, \delta}(f(\zeta))^{-1} \Delta_{j}(f(\zeta)) \Upsilon_{j, \delta}(f(\zeta))^{-1} \right ]_{r+1, r+1} \frac{f'(\zeta)}{f(\zeta)} \dif \zeta, \\
	t_{3} &= \lim_{\delta \downarrow 0} \frac{1}{2\pi \i} \frac{1+ \delta}{\delta^{2}} \oint_{\Gamma_{i}} \left [ \Upsilon_{j, \delta}(f(\zeta))^{-1} \Delta_{j}(f(\zeta)) \Upsilon_{j, \delta}(f(\zeta))^{-1} \Delta_{j}(f(\zeta)) \Upsilon_{j, \delta}(f(\zeta))^{-1} \right ]_{r+1, r+1} \frac{f'(\zeta)}{f(\zeta)} \dif \zeta, \\
	t_{4} &= \lim_{\delta \downarrow 0} \frac{1}{2\pi \i} \frac{1+ \delta}{\delta^{2}} \oint_{\Gamma_{i}} \bigg [ \Upsilon_{j, \delta}(f(\zeta))^{-1} \Delta_{j}(f(\zeta)) \Upsilon_{j, \delta}(f(\zeta))^{-1} \Delta_{j}(f(\zeta)) \\
	& \qquad \qquad \qquad \qquad \qquad \qquad \times \left ( \sfD_{j, \delta}^{-1} + \phi^{-1/2} V_{j} \widetilde{G}(f(\zeta)) V_{j} \right )^{-1} \Delta_{j}(f(\zeta)) \Upsilon_{j, \delta}(z)^{-1} \bigg ]_{r+1, r+1} \frac{f'(\zeta)}{f(\zeta)} \dif \zeta.
\end{align*}
By similar arguments that we used to show $ s_{1} = 0 $ above, we can show that $ \left [ \Upsilon_{j, \delta}(f(\zeta))^{-1} \right ]_{r+1, r+1} $ is holomorphic inside $ \Gamma_{i} $ for sufficiently small $ \delta > 0 $ from \eqref{eq:diagonal of Upsilon inverse 2}, and thus we have 
\begin{equation} \label{eq:t_1}
	t_{1} = 0.
\end{equation}
For $ t_{2} $, since
\[ \left [ \Upsilon_{j, \delta}(f(\zeta))^{-1} \Delta_{j}(z) \Upsilon_{j, \delta}(f(\zeta))^{-1} \right ]_{r+1, r+1} = \left [ \Upsilon_{j, \delta}(f(\zeta))^{-1} \right ]_{r+1, r+1}^{2} [ \Delta_{j}(f(\zeta)) ]_{r+1, r+1} \]
and $ \frac{\phi^{-1/2} \zeta \sigma_{j}}{1 + \phi^{-1/2} \zeta \sigma_{j}} $ is holomorphic in $ f(\Gamma_{i}) $ for sufficiently small $ \delta > 0 $, we also have
\begin{equation} \label{eq:t_2}
	t_{2} = 0.
\end{equation}

To estimate $ t_{3} $, notice that
\[ \left [ \Delta_{j}(f(\zeta)) \Upsilon_{j, \delta}(f(\zeta))^{-1} \Delta_{j}(f(\zeta)) \right ]_{r+1, r+1} = \sum_{k=1}^{r+1} \frac{[\Delta_{j}(f(\zeta))]_{r+1, k} [\Delta_{j}(f(\zeta))]_{k, r+1}}{[\Upsilon_{j, \delta}(f(\zeta))^{-1}]_{k,k}}. \]
Then, $ t_{3} $ can be rewritten as
\begin{align*}
	t_{3} &= \lim_{\delta \downarrow 0} \frac{1}{2\pi \i} \frac{1+ \delta}{\delta^{2}} \oint_{\Gamma_{i}} [\Upsilon_{j, \delta}(f(\zeta))^{-1}]_{r+1,r+1}^{2} \sum_{k=1}^{r+1} \frac{[\Delta_{j}(f(\zeta))]_{r+1, k} [\Delta_{j}(f(\zeta))]_{k, r+1}}{[\Upsilon_{j, \delta}(f(\zeta))]_{k,k}} \frac{f'(\zeta)}{f(\zeta)} \dif \zeta.
\end{align*}
Again, from the holomorphicity of $ {[\Upsilon_{j, \delta}(f(\zeta))^{-1}]_{k,k}} $ for $ k \neq i $, Cauchy's integral theorem gives that
\begin{align}
	t_{3} &= \lim_{\delta \downarrow 0} \frac{1}{2\pi \i} \frac{1+ \delta}{\delta^{2}} \oint_{\Gamma_{i}} [\Upsilon_{j, \delta}(f(\zeta))^{-1}]_{r+1,r+1}^{2} \frac{[\Delta_{j}(f(\zeta))]_{r+1, i} [\Delta_{j}(f(\zeta))]_{i, r+1}}{[\Upsilon_{j, \delta}(f(\zeta))]_{i,i}} \frac{f'(\zeta)}{f(\zeta)} \dif \zeta \notag \\
	&= \lim_{\delta \downarrow 0} \frac{1}{2\pi \i} \frac{1+ \delta}{\delta^{2}} \oint_{\Gamma_{i}} \Bigg [ \left ( \frac{\delta}{1+\delta} \cdot \frac{\zeta + \phi^{1/2} \sigma_{j}^{-1}}{\zeta + \phi^{1/2} \widetilde{\sigma}_{j}^{-1}} \right )^{2} \left ( \frac{\sfd_{i}}{1+\sfd_{i}} \cdot \frac{\zeta + \phi^{1/2} \sigma_{i}^{-1}}{\zeta + \phi^{1/2} \widetilde{\sigma}_{i}^{-1}} \right ) \notag \\
	& \qquad \qquad \qquad \qquad \qquad \times \phi^{-1} \langle \bfv_{j}, \widetilde{G}(f(\zeta)) \bfv_{i} \rangle \langle \bfv_{i}, \widetilde{G}(f(\zeta)) \bfv_{j} \rangle \frac{f'(\zeta)}{f(\zeta)} \Bigg ] \dif \zeta \notag \\
	&= \frac{\mathsf{d}_{i}^{2} \mathfrak{b}_{i}}{\phi (1+\sfd_{i})} \langle \bfv_{j}, \widetilde{G}(\mathfrak{a}_{i}) \bfv_{i} \rangle \langle \bfv_{i}, \widetilde{G}(\mathfrak{a}_{i}) \bfv_{j} \rangle, \label{eq:t_3}
\end{align}
where we use the residue theorem for the last step.

Finally, we establish an upper bound for $ |t_{4}| $. From \eqref{eq:estimate of m} and \eqref{eq:diagonal of Upsilon inverse 2}, 
\begin{equation} \label{eq:outlier eigenvector 2 error proof 1}
	\left [ \Upsilon_{j, \delta}(f(\zeta))^{-1} \right ]_{r+1, r+1} = \mathrm{O}(\delta)
\end{equation}
for $ \zeta \in \Gamma_{i} $ and sufficiently small $ \delta $. Furthermore, from \eqref{eq:diagonal of Upsilon inverse 2}, and \eqref{eq:outlier eigenvector 1 error proof 1}, there exists some constant $ C > 0 $ such that
\begin{equation} \label{eq:outlier eigenvector 2 error proof 2}
	\| \Upsilon_{j, \delta}(f(\zeta))^{-1} \| \le C \phi^{1/2}, \quad \left \| \left ( \sfD_{j, \delta}^{-1} + \phi^{-1/2} V_{j} \widetilde{G}(f(\zeta)) V_{j} \right )^{-1} \right \| \le C \phi^{1/2}
\end{equation}
for $ \zeta \in \Gamma_{i} $ and sufficiently small $ \delta $.
\eqref{eq:outlier eigenvector 2 proof 2}, \eqref{eq:outlier eigenvector 2 error proof 1}, \eqref{eq:outlier eigenvector 2 error proof 2}, and the basic estimates $ f'(\zeta) \asymp 1 $ and $ f(\zeta) \asymp \phi^{1/2} $ give that
\begin{equation} \label{eq:t_4}
	|t_{4}| \le C \phi^{-1} n^{-3/2 + 3 \epsilon} = C p^{-1} n^{-1/2 + 3 \epsilon}.
\end{equation}
By \eqref{eq:t_1}, \eqref{eq:t_2}, \eqref{eq:t_3}, and \eqref{eq:t_4}, we complete the proof of \eqref{eq:outlier eigenvector stronger}.

\noindent\textbf{Proof of Theorem \ref{thm:eigenvector} (ii)}. We first consider the case $ j \in \llb r \rrb $. Fix a sufficiently small $ \tau > 0 $. By , there exists a high probability event $ \Xi_{5} $ such that
\begin{gather}
	\bbone(\Xi_{5}) \left | \widetilde{\lambda}_{i} - \gamma_{i-r} \right | \prec ( \min \{ i, n+1-i \} )^{-1/3} n^{-2/3 + \epsilon}, \label{eq:outlier delocalization proof 1} \\
	\bbone(\Xi_{5}) \| V_{\mathrm{s}}^{\top} (\widetilde{G}(z) - \Pi_{\widetilde{G}}(z)) V_{\mathrm{s}} \| \le n^{\epsilon} \left ( \sqrt{\frac{\im m(z)}{n\eta}} + \frac{1}{n\eta} \right ) \quad \text{for all } z \in \bfD \equiv \bfD(\tau). \label{eq:outlier delocalization proof 2}
\end{gather}
We fix a realization in $ \Xi_{5} $.

Consider a strictly increasing mapping $ \eta \mapsto \eta \im m(\widetilde{\lambda}_{i} + \i \eta) $. Since $ \eta \im m (\widetilde{\lambda}_{i} + \i \eta) \asymp 1 $ for $ \eta \asymp 1 $ and $ n^{-1 + \epsilon} \im m(\widetilde{\lambda}_{i} + \i n^{-1 + \epsilon}) \lesssim n^{-1 + \epsilon} $ by Lemma \ref{lem:basic estimates of m} and \eqref{eq:outlier delocalization proof 1}, we can choose $ \eta_{i} > 0 $ such that 
\begin{equation} \label{eq:choice of eta_i}
	\im m(\widetilde{\lambda}_{i} + \i \eta_{i}) = \frac{n^{6\epsilon}}{n\eta_{i}}.
\end{equation}
Let $ z_{i} = \widetilde{\lambda}_{i} + \i\eta_{i} $. Then, \eqref{eq:outlier delocalization proof 2} for $ z_{i} $ can be rewritten as
\begin{equation} \label{eq:outlier delocalization proof 3}
	\left \| V_{\mathrm{s}}^{\top} \left ( \widetilde{G}(z_{i}) - \Pi_{\widetilde{G}}(z_{i}) \right ) V_{\mathrm{s}} \right \| \le \frac{n^{4\epsilon}}{n \eta}.
\end{equation}

As in \eqref{eq:delocalization proof}, we have 
\begin{equation} \label{eq:outlier delocalizatin ineq}
	\langle \widetilde{\bfu}_{i}, \bfv_{j} \rangle^{2} \le \eta_{i} \im \left ( \bfv_{j}^{\top} \widetilde{R}(z_{i}) \bfv_{j} \right ).
\end{equation}
Recall the definitions of $ \Upsilon $ and $ \Delta $ in \eqref{eq:outlier eigenvector 1 leading} and \eqref{eq:outlier eigenvector 1 error}. From Lemma \ref{lem:perturb 2}, \eqref{eq:diagonal of Upsilon inverse}, and the resolvent expansion
\[ (\Upsilon - \Delta)^{-1} = \Upsilon^{-1} + \Upsilon^{-1} \Delta \Upsilon^{-1} + \Upsilon^{-1} \Delta (\Upsilon - \Delta)^{-1} \Delta \Upsilon^{-1}, \]
we have that
\begin{equation}
	\begin{aligned}
		\bfv_{j}^{\top} \widetilde{R}(z_{i}) \bfv_{j} 
		&= \frac{1}{z_{i}} \left [ \frac{1}{\sfd_{j}} - \frac{1 + \sfd_{j}}{\sfd_{j}^{2}} \left ( \frac{\sfd_{j}}{1+\sfd_{j}} \cdot \frac{m(z_{i}) + \phi^{1/2} \sigma_{j}^{-1}}{m(z_{i}) + \phi^{1/2} \widetilde{\sigma}_{j}^{-1}} + \left ( \frac{\sfd_{j}}{1+\sfd_{j}} \cdot \frac{m(z_{i}) + \phi^{1/2} \sigma_{j}^{-1}}{m(z_{i}) + \phi^{1/2} \widetilde{\sigma}_{j}^{-1}} \right )^{2} \left ( \Delta + \Delta (\Upsilon - \Delta)^{-1} \Delta \right )_{jj} \right ) \right ] \\
		&= \frac{1}{z_{i}} \left [ -\frac{\phi^{1/2} \widetilde{\sigma}_{j}^{-1}}{m(z_{i}) + \phi^{1/2} \widetilde{\sigma}_{j}^{-1}} - \frac{1}{1+\sfd_{j}} \left ( \frac{m(z_{i}) + \phi^{1/2} \sigma_{j}^{-1}}{m(z_{i}) + \phi^{1/2} \widetilde{\sigma}_{j}^{-1}} \right )^{2} \left ( \Delta + \Delta (\Upsilon - \Delta)^{-1} \Delta \right )_{jj} \right ].
	\end{aligned}
\end{equation}
Combining this with \eqref{eq:outlier delocalizatin ineq}, we have
\begin{equation} \label{eq:outlier delocalization proof terms}
	\begin{aligned}
		\langle \widetilde{\bfu}_{i}, \bfv_{j} \rangle^{2} &\le \frac{\eta_{i}^{2}}{|z_{i}|^{2}} \re \left ( \frac{\phi^{1/2} \widetilde{\sigma}_{j}^{-1}}{m(z_{i}) + \phi^{1/2} \widetilde{\sigma}_{j}^{-1}} \right ) - \frac{\widetilde{\lambda}_{i} \eta_{i}}{|z_{i}|^{2}} \im \left ( \frac{\phi^{1/2} \widetilde{\sigma}_{j}^{-1}}{m(z_{i}) + \phi^{1/2} \widetilde{\sigma}_{j}^{-1}} \right ) \\
		& \qquad \qquad \qquad + \frac{\eta_{i} | m(z_{i}) + \phi^{1/2} \sigma_{j}^{-1} |^{2}}{|z_{i}| (1 + \sfd_{j}) | m(z_{i}) + \phi^{1/2} \widetilde{\sigma}_{j}^{-1} |^{2}} \left \| \Delta + \Delta (\Upsilon - \Delta)^{-1} \Delta \right \|.
	\end{aligned}
\end{equation}
From Assumption \ref{assump:main} (iv), the estimate $ |z_{i}| \asymp |\widetilde{\lambda}_{i}| \asymp \phi^{1/2} $, and \eqref{eq:choice of eta_i}, the first term of \eqref{eq:outlier delocalization proof terms} satisfies
\begin{equation} \label{eq:outlier delocalization proof term 1}
	\frac{\eta_{i}^{2}}{|z_{i}|^{2}} \re \left ( \frac{\phi^{1/2} \widetilde{\sigma}_{j}^{-1}}{m(z_{i}) + \phi^{1/2} \widetilde{\sigma}_{j}^{-1}} \right ) \le \frac{\eta_{i}^{2}}{|z_{i}|^{2}} \left | \frac{\phi^{1/2} \widetilde{\sigma}_{j}^{-1}}{m(z_{i}) + \phi^{1/2} \widetilde{\sigma}_{j}^{-1}} \right | \le C \phi^{-1} \frac{\eta_{i}^{2}}{\im m(z_{i})} \le C \phi^{-1} n \eta_{i}^{3} \le C p^{-1} n^{12\epsilon},
\end{equation}
where we use that $ \eta_{i} \le n^{-2/3 + 4\epsilon} $ for the last step, which can verified from \ref{lem:basic estimates of m} and \eqref{eq:choice of eta_i}.
Similarly, we have that
\begin{equation} \label{eq:outlier delocalization proof term 2}
	- \frac{\widetilde{\lambda}_{i} \eta_{i}}{|z_{i}|^{2}} \im \left ( \frac{\phi^{1/2} \widetilde{\sigma}_{j}^{-1}}{m(z_{i}) + \phi^{1/2} \widetilde{\sigma}_{j}^{-1}} \right ) \le C\frac{\widetilde{\lambda}_{i} \eta_{i} \im m(z_{i})}{|z_{i}|^{2}} \le C \phi^{-1/2} \eta_{i} \im m(z_{i}) = C (pn)^{-1/2} n^{6\epsilon}.
\end{equation}
It remains to estimate the third term of \eqref{eq:outlier delocalization proof terms}. From \eqref{eq:choice of eta_i} and \eqref{eq:outlier delocalization proof 3},
\begin{equation} \label{eq:outlier delocalization proof term 3 1}
	\min_{k \in \llb r \rrb} \left | [\Upsilon(z_{i})]_{kk} \right | = \min_{k \in \llb r \rrb} \left | \sfd_{k}^{-1} + \frac{\phi^{-1/2} m(z_{i}) \sigma_{k}}{1 + \phi^{-1/2} m(z_{i}) \sigma_{k}} \right | \gtrsim \phi^{-1/2} \im m(z_{i}) = \phi^{-1/2} \frac{n^{6\epsilon}}{n\eta_{i}} \gg \phi^{-1/2} \frac{n^{4\epsilon}}{n\eta_{i}} \ge \| \Delta \|.
\end{equation}
Then,
\begin{equation} \label{eq:outlier delocalization proof term 3 2}
	\| (\Upsilon - \Delta)^{-1} \| \le \| (I -\Upsilon^{-1} \Delta)^{-1} \| \| \Upsilon^{-1} \| \le C \phi^{1/2} \frac{n\eta_{i}}{n^{6\epsilon}}.
\end{equation}
By Assumption \ref{assump:main} (iii)-(iv), the estimate $ |z_{i}| \asymp \phi^{1/2} $, \eqref{eq:outlier delocalization proof term 3 2}, and \eqref{eq:outlier delocalization proof term 3 2}, we have
\begin{equation} \label{eq:outlier delocalization proof term 3}
	\frac{\eta_{i} | m(z_{i}) + \phi^{1/2} \sigma_{j}^{-1} |^{2}}{|z_{i}| (1 + \sfd_{j}) | m(z_{i}) + \phi^{1/2} \widetilde{\sigma}_{j}^{-1} |^{2}} \left \| \Delta + \Delta (\Upsilon - \Delta)^{-1} \Delta \right \| \le C (pn)^{-1/2} n^{4\epsilon}.
\end{equation}
By \eqref{eq:outlier delocalization proof term 1}, \eqref{eq:outlier delocalization proof term 2}, and \eqref{eq:outlier delocalization proof term 3}, we conclude that
\[ \langle \widetilde{\bfu}_{i}, \bfv_{j} \rangle^{2} \prec (pn)^{-1/2}. \]

The case where $ j \in \llb r+1, p \rrb $ can be handled by the same perturbation argument as in the proof of \eqref{eq:outlier eigenvector}. In this case, since $ \widetilde{\sigma}_{j} = \sigma_{j} \asymp 1 $, the scaling differs from that for $ j \in \llb r \rrb $ and yields a bound of $ p^{-1} $. We omit the details.


\subsection{Proof of Theorem \ref{thm:spiked outlier eigenvector projection sum}} \label{app:spiked outlier eigenvector projection sum proof}
We adapt the argument from the proof of Theorem 2.4 in \cite{dly24}. Fix $ i \in \llb r \rrb $. By \eqref{eq:outlier eigenvector stronger} and that $ \ell_{j} $ is bounded from below and above,
\begin{align}
	\sum_{j=r+1}^{p} \ell_{j} \langle \widetilde{\bfu}_{i}, \bfv_{j} \rangle^{2} &= \frac{\mathsf{d}_{i}^{2} \mathfrak{b}_{i}}{\phi (1+\sfd_{i})} \sum_{j=r+1}^{p} \ell_{j} \langle \bfv_{j}, \widetilde{G}(\mathfrak{a}_{i}) \bfv_{i} \rangle \langle \bfv_{i}, \widetilde{G}(\mathfrak{a}_{i}) \bfv_{j} \rangle + \mathrm{O}_{\prec}(n^{-1/2}) \notag \\
	&= \frac{\mathsf{d}_{i}^{2} \mathfrak{b}_{i}}{\phi (1+\sfd_{i})} \left [ \sum_{j=1}^{p} \ell_{j} \langle \bfv_{j}, \widetilde{G}(\mathfrak{a}_{i}) \bfv_{i} \rangle \langle \bfv_{i}, \widetilde{G}(\mathfrak{a}_{i}) \bfv_{j} \rangle - \sum_{j=1}^{r} \ell_{j} \langle \bfv_{j}, \widetilde{G}(\mathfrak{a}_{i}) \bfv_{i} \rangle \langle \bfv_{i}, \widetilde{G}(\mathfrak{a}_{j}) \bfv_{j} \rangle \right ] + \mathrm{O}_{\prec}(n^{-1/2}) \notag \\
	&= \frac{\mathsf{d}_{i}^{2} \mathfrak{b}_{i}}{\phi (1+\sfd_{i})} \left [ \sum_{j=1}^{p} \ell_{j} \langle \bfv_{j}, \widetilde{G}(\mathfrak{a}_{i}) \bfv_{i} \rangle \langle \bfv_{i}, \widetilde{G}(\mathfrak{a}_{i}) \bfv_{j} \rangle - \ell_{i} \left ( \frac{m(\mathfrak{a}_{i}) \sigma_{i}}{1 + \phi^{-1/2} m(\mathfrak{a}_{i}) \sigma_{i}} \right )^{2} \right ] + \mathrm{O}_{\prec}(n^{-1/2}), \label{eq:sum proof 1}
\end{align}
where we use Lemma \ref{lem:G local law} (ii) for the last step. From the definition of $ \widetilde{G} $ in \eqref{eq:G and counterpart of G} and Lemma \ref{lem:local law outside}, the first term can be rewritten as
\begin{align}
	& \sum_{j=1}^{p} \ell_{j} \langle \bfv_{j}, \widetilde{G}(\mathfrak{a}_{i}) \bfv_{i} \rangle \langle \bfv_{i}, \widetilde{G}(\mathfrak{a}_{j}) \bfv_{j} \rangle \\
	= & \ \phi \ell_{i} + 2 \phi \mathfrak{a}_{i} \ell_{i} \langle \bfv_{i}, R_{0}(\mathfrak{a}_{i}) \bfv_{i} \rangle + \phi \mathfrak{a}_{i}^{2} \sum_{j=1}^{p} \ell_{j} \langle \bfv_{j}, R_{0}(\mathfrak{a}_{i}) \bfv_{i} \rangle \langle \bfv_{i}, R_{0}(\mathfrak{a}_{i}) \bfv_{j} \rangle \notag \\
	= & \ \phi \ell_{i} \left ( 1 - \frac{2}{1 + \phi^{-1/2} m (\fa_{i}) \sigma_{i}} \right ) + \phi \mathfrak{a}_{i}^{2} \sum_{j=1}^{p} \ell_{j} \langle \bfv_{j}, R_{0}(\mathfrak{a}_{i}) \bfv_{i} \rangle \langle \bfv_{i}, R_{0}(\mathfrak{a}_{i}) \bfv_{j} \rangle + \mathrm{O}_{\prec}(n^{-1/2}). \label{eq:sum proof 2}
\end{align}
We denote the sum by
\[ \cL_{i} := \sum_{j=1}^{p} \ell_{j} \langle \bfv_{j}, R_{0}(\mathfrak{a}_{i}) \bfv_{i} \rangle \langle \bfv_{i}, R_{0}(\mathfrak{a}_{i}) \bfv_{j} \rangle. \]

For $ \sft > 0 $, we define two resolvents
\begin{align}
	\cR_{\sft}(z) &\equiv \cR_{\sft}(z ; \ell) := (\Sigma_{0}^{1/2} XX^{\top} \Sigma_{0}^{1/2} - z - \phi^{1/2} \sft VLV^{\top})^{-1}, \\
	\sfR_{\sft,\fa_{i}}(z) &\equiv \sfR_{\sft, \fa_{i}}(z;\ell) := (\Sigma_{\sft, \fa_{i}, \ell}^{1/2} XX^{\top} \Sigma_{\sft, \fa_{i}, \ell}^{1/2} - z)^{-1},
\end{align}
where $ L = \operatorname{diag}(\ell_{1}, \dots, \ell_{p}) $ and $ \Sigma_{\sft, \fa_{i}, \ell} = V \Lambda_{\sft, \fa_{i}, \ell} V^{\top} $ with $ \Lambda_{\sft, \fa_{i}, \ell} = (1 + \phi^{1/2} \sft L/\mathfrak{a}_{i})^{-1} \Lambda_{0} $. Then, we have that
\begin{equation} \label{eq:sum proof L_i}
	\cL_{i} = \phi^{-1/2} \frac{\partial \langle \bfv_{i}, \cR_{\sft} (\mathfrak{a}_{i}) \bfv_{i} \rangle}{\partial \sft} \Bigg |_{\sft = 0}
\end{equation}
and
\begin{equation} \label{eq:sum proof R_t(a_i)}
	\cR_{\sft}(\mathfrak{a}_{i}) = V \left ( 1 + \frac{\phi^{1/2} \sft L}{\mathfrak{a}_{i}} \right )^{-1/2} V^{\top} \sfR_{\sft, \fa_{i}}(\mathfrak{a}_{i}) V \left ( 1 + \frac{\phi^{1/2} \sft L}{\mathfrak{a}_{i}} \right )^{-1/2} V^{\top}.
\end{equation}

We now let $ \sft = \mathrm{o}(1) $ and let $ \sfm_{\sft, \fa_{i}}(z) \equiv \sfm_{\sft, \mathfrak{a}_{i}}(z;\ell) $ denote the Stieltjes transform corresponding to $ \Sigma_{\sft,\fa_{i}, \ell} $, defined in the same manner as $ m $ is associated with $ \Sigma_{0} $, that is, $ \sfm_{\sft, \fa_{i}}(z) \in \bbC^{+} $ is the unique solution to the following self-consistent equation
\begin{equation} \label{eq:self-consistent equation variation}
	z = -\frac{1}{\sfm_{\sft, \fa_{i}}(z)} + \frac{\phi^{1/2}}{p} \sum_{j=1}^{p} \frac{(1 + \phi^{1/2} \sft \ell_{j}/\fa_{i})^{-1} \sigma_{j}}{1 + \phi^{-1/2} \sfm_{\sft, \fa_{i}}(z) (1 + \phi^{1/2} \sft \ell_{j}/\fa_{i})^{-1} \sigma_{j}}.
\end{equation}
Since $ \{ (1 + \phi^{1/2} \sft L/ \fa_{i})^{-1} \sigma_{j} \}_{j=1}^{p} $ satisfies Assumption \ref{assump:main} (iii), the local law \eqref{eq:local law outside} also holds for $ \sfR_{\sft, \fa_{i}} $ by substituting $ \Lambda_{0} $ and $ m(z) $ with $ \Lambda_{\sft, \fa_{i}, \ell} $ and $ \sfm_{\sft, \fa_{i}}(z) $, respectively. Together with \eqref{eq:sum proof R_t(a_i)}, this leads to
\begin{align}
	\langle \bfv_{i}, \cR_{\sft} (\mathfrak{a}_{i}) \bfv_{i} \rangle &= \left ( 1 + \frac{\phi^{1/2} \sft \ell_{i}}{\mathfrak{a}_{i}} \right )^{-1} \langle \bfv_{i}, \sfR_{\sft, \fa_{i}}(\fa_{i}) \bfv_{i} \rangle \notag \\
	&= \left ( 1 + \frac{\phi^{1/2} \sft \ell_{i}}{\mathfrak{a}_{i}} \right )^{-1} \left ( \frac{-1}{\fa_{i} ( 1 + \phi^{-1/2} \sfm_{\sft, \fa_{i}}(\fa_{i}) (1 + \phi^{1/2} \sft \ell_{i}/\mathfrak{a}_{i})^{-1} \sigma_{i})} \right ) + \mathrm{O}_{\prec}(\phi^{-1} n^{-1/2}) \notag \\
	&= \frac{-1}{\fa_{i} + \phi^{1/2} \sft \ell_{i} + \fa_{i} \phi^{-1/2} \sfm_{\sft, \fa_{i}}(\fa_{i}) \sigma_{i}} + \mathrm{O}_{\prec}(\phi^{-1} n^{-1/2}) \label{eq:sum proof vRv}
\end{align}
Furthermore, it follows from Lemma \ref{lem:rigidity}, Assumption \ref{assump:main} (iv), and \eqref{eq:sum proof L_i} that $ \| \cR_{\sft}(\fa_{i}) \| \lesssim 1 $ with high probability. Hence, we have
\begin{equation} \label{eq:sum proof L_i mvt}
	\cL_{i} = \frac{\langle \bfv_{i}, \cR_{\sft} (\mathfrak{a}_{i}) \bfv_{i} \rangle - \langle \bfv_{i}, \cR_{0} (\mathfrak{a}_{i}) \bfv_{i} \rangle}{\phi^{1/2}\sft} + \mathrm{O}(\phi^{-1/2} \sft).
\end{equation}
Let $ \sft = \phi^{-1/2} n^{-1/4} $. Combining \eqref{eq:sum proof vRv} and \eqref{eq:sum proof L_i mvt}, we have
\begin{align}
	\cL_{i} &= - \frac{\partial}{\partial \sft} \frac{\phi^{-1/2}}{\fa_{i} + \phi^{1/2} \sft \ell_{i} + \fa_{i} \phi^{-1/2} \sfm_{\sft, \fa_{i}}(\fa_{i}) \sigma_{i}} + \mathrm{O}_{\prec} \left ( \phi^{-3/2} n^{-1/2}/\sft + \phi^{-1/2} \sft \right ) \\
	&= \frac{\ell_{i} + \fa_{i} \phi^{-1} \sigma_{i} \dot{\sfm}_{0, \fa_{i}}(\fa_{i})}{[\fa_{i}+ \fa_{i} \phi^{-1/2} \sigma_{i} m(\fa_{i})]^{2}} + \mathrm{O}(\phi^{-1} n^{-1/4}), \label{eq:sum proof 3}
\end{align}
where $ \dot{\sfm}_{\sft, \fa_{i}}(\fa_{i}) = \frac{\partial}{\partial \sft} \sfm_{\sft, \fa_i}(\fa_{i}) $. From \eqref{eq:estimate of m}, \eqref{eq:sum proof 1}, \eqref{eq:sum proof 2}, and \eqref{eq:sum proof 3}, we conclude that
\[ \sum_{j=r+1}^{p} \ell_{j} \langle \widetilde{\bfu}_{i}, \bfv_{j} \rangle^{2} = \phi^{-1} \fa_{i} \fb_{i} \widetilde{\sigma}_{i} \dot{\sfm}_{0, \fa_{i}}(\fa_{i}) + \mathrm{O}_{\prec}(n^{-1/4}). \]

Finally, we compute $ \dot{\sfm}_{\sft, \fa_{i}}(\fa_{i}) $. By taking partial derivatives to \eqref{eq:self-consistent equation variation} with respect to $ z $ and $ \sft $, we have
\begin{align*}
	1 &= \frac{\sfm_{\sft, \fa_{i}}'(z)}{\sfm_{\sft, \fa_{i}}(z)^{2}} - \frac{1}{p} \sum_{j=1}^{p} \frac{\sfm_{\sft, \fa_{i}}'(z)}{[(1 + \phi^{1/2} \sft \ell_{j}/\fa_{i}) \sigma_{j}^{-1} + \phi^{-1/2} \sfm_{\sft, \fa_{i}}(z)]^{2}}, \\
	0 &= \frac{\dot{\sfm}_{\sft, \fa_{i}}(z)}{\sfm_{\sft, \fa_{i}}(z)^{2}} - \frac{\phi^{1/2}}{p} \sum_{j=1}^{p} \frac{\phi^{1/2} \ell_{j}/(\fa_{i} \sigma_{j}) + \phi^{-1/2} \dot{\sfm}_{\sft, \fa_{i}}(z)}{[(1 + \phi^{1/2} \sft \ell_{j}/\fa_{i}) \sigma_{j}^{-1} + \phi^{-1/2} \sfm_{\sft, \fa_{i}}(\fa_{i})]^{2}}.
\end{align*}
Solving the equations gives that
\begin{equation*}
	\dot{\sfm}_{\sft, \fa_{i}}(\fa_{i}) = \frac{\phi \sfm'_{\sft, \fa_{i}}(\fa_{i})}{p \fa_{i}} \sum_{j=1}^{p} \frac{\ell_{j}/ \sigma_{j}}{[(1 + \phi^{1/2} \sft \ell_{j}/\fa_{i}) \sigma_{j}^{-1} + \phi^{-1/2} \sfm_{\sft, \fa_{i}}(z)]^{2}}.
\end{equation*}
With $ \sft = 0 $,
\begin{equation*}
	\dot{\sfm}_{0, \fa_{i}}(\fa_{i}) = \frac{\phi m'(\fa_{i})}{p \fa_{i}} \sum_{j=1}^{p} \frac{\ell_{j}/ \sigma_{j}}{[\sigma_{j}^{-1} + \phi^{-1/2} m(\fa_{i})]^{2}},
\end{equation*}
which is given in \eqref{eq:mdot}.

\bibliographystyle{plain}
\bibliography{ref}

\end{document}